\documentclass[11pt]{article}

\usepackage{amssymb}
\usepackage{amsmath}
\usepackage{amsthm}
\usepackage{accents}
\usepackage{enumerate}
\usepackage{booktabs}
\usepackage{algorithm}
\usepackage{graphicx}
\usepackage{url}
\usepackage{bm}
\usepackage{color}
\usepackage{multirow}

\theoremstyle{plain}
\newtheorem{theo}{Theorem}
\newtheorem{prop}{Proposition}
\newtheorem{lemm}{Lemma}
\newtheorem{coro}{Corollary}
\newtheorem{assump}{Assumption}

\theoremstyle{definition}
\newtheorem{remark}{Remark}
\newtheorem{prob}{Problem}

\def\0{\bm{0}}
\def\1{\bm{1}}
\def\2{\bm{2}}
\def\3{\bm{3}}
\def\4{\bm{4}}
\def\5{\bm{5}}
\def\6{\bm{6}}
\def\7{\bm{7}}
\def\8{\bm{8}}
\def\9{\bm{9}}

\def\a{\bm{a}}
\def\b{\bm{b}}
\def\d{\bm{d}}
\def\e{\bm{e}}
\def\f{\bm{f}}
\def\h{\bm{h}}

\def\m{\bm{m}}
\def\n{\bm{n}}
\def\p{\bm{p}}
\def\s{\bm{s}}
\def\t{\bm{t}}
\def\u{\bm{u}}
\def\v{\bm{v}}

\def\x{\bm{x}}

\def\A{\bm{A}}
\def\B{\bm{B}}
\def\C{\bm{C}}
\def\D{\bm{D}}
\def\E{\bm{E}}
\def\F{\bm{F}}
\def\G{\bm{G}}
\def\H{\bm{H}}
\def\I{\bm{I}}

\def\L{\bm{L}}
\def\M{\bm{M}}
\def\N{\bm{N}}
\def\P{\bm{P}}
\def\Q{\bm{Q}}

\def\S{\bm{S}}

\def\U{\bm{U}}
\def\V{\bm{V}}
\def\W{\bm{W}}
\def\X{\bm{X}}
\def\Y{\bm{Y}}
\def\Z{\bm{Z}}

\def\IC{\mathcal{I}}

\def\SC{\mathcal{S}}

\def\Real{\mathbb{R}}

\def\Pib{\bm{\Pi}}
\def\Sigmab{\bm{\Sigma}}
\def\Omegab{\bm{\Omega}}

\def\mmin{\mbox{\scriptsize min}}
\def\mmax{\mbox{\scriptsize max}}

\makeatletter
\def\widebar{\accentset{{\cc@style\underline{\mskip10mu}}}}
\def\Widebar{\accentset{{\cc@style\underline{\mskip8mu}}}}
\makeatother

\def\wb{\widebar}
\def\wh{\widehat}

\def\MVEE{\mathrm{P}}

\def\spaApprox{\mathsf{spaApprox}}
\def\randApprox{\mathsf{randApprox}}
\def\svdApprox{\mathsf{svdApprox}}

\def\spa{\mathsf{spa}}

\def\pspa{\mathsf{pspa}}
\def\mpspa{\mathsf{mpspa}}

\def\erspa{\mathsf{erspa}}
\def\merspa{\mathsf{merspa}}

\def\spaspa{\mathsf{spaspa}}
\def\vca{\mathsf{vca}}

\newcommand{\by}[2]{$#1 \times #2$}

\title{Efficient Preconditioning for Noisy Separable NMFs by
Successive Projection Based Low-Rank Approximations}

\author{Tomohiko~Mizutani%
\thanks{Department of Industrial Engineering and Economics,
Tokyo Institute of Technology,
2-12-1-W9-69, Ookayama, Meguro-ku, Tokyo, 152-8552, Japan. 
{\tt mizutani.t.ab@m.titech.ac.jp}}
\and
Mirai~Tanaka%
\thanks{Department of Mathematical Analysis and Statistical Inference,
The Institute of Statistical Mathematics, 
10-3, Midori-cho, Tachikawa, Tokyo 190-8562, Japan.
{\tt mirai@ism.ac.jp}}}

\date{\today}

\begin{document}

\maketitle

\begin{abstract}
 The successive projection algorithm (SPA) can quickly solve
 a nonnegative matrix factorization problem under a separability assumption.
 Even if noise is added to the problem,
 SPA is robust as long as the perturbations caused by the noise are small.
 In particular, robustness against noise should be high
 when handling the problems arising from real applications.
 The preconditioner proposed by Gillis and Vavasis (2015)
 makes it possible to enhance the noise robustness of SPA.
 Meanwhile, an additional computational cost is required.
 The construction of the preconditioner contains a step
 to compute the top-$k$ truncated singular value decomposition of an input matrix.
 It is known that the decomposition provides the best rank-$k$ approximation to the input matrix;
 in other words, a matrix with the smallest approximation error
 among all matrices of rank less than $k$.
 This step is an obstacle to  an efficient implementation of
 the preconditioned SPA.

 To address the cost issue,
 we propose a modification of the algorithm for constructing the preconditioner.
 Although the original algorithm uses the best rank-$k$ approximation,
 instead of it, our modification uses an alternative.
 Ideally, this alternative should have high approximation accuracy and low computational cost.
 To ensure this, our modification employs a rank-$k$ approximation
 produced by an SPA based algorithm.
 We analyze the accuracy of the approximation and evaluate the computational cost of the algorithm.
 We then present an empirical study revealing
 the actual performance of the SPA based rank-$k$ approximation algorithm
 and the modified preconditioned SPA.

 \end{abstract}

\section{Introduction} \label{Sec: intro}
Given $\M \in \Real^{d \times m}_+$ and a positive integer $k$,
a nonnegative matrix factorization (NMF) problem
is one of finding $\F \in \Real^{d \times k}_+$ and $\W \in \Real^{k \times m}_+$
such that $\|\F\W - \M\|_F$ is minimized.
Here, a nonnegative matrix is a real matrix whose elements are all nonnegative,
and $\Real^{d \times m}_+$ denotes the set of \by{d}{m} nonnegative matrices.
Although NMF problems cover a broad range of applications,
they are often intractable.
Indeed, the NMF problem was shown to be NP-hard in \cite{Vav09};
also see \cite{Aro12a} for a further discussion on the hardness of the problem.
However, the situation changes
if we make the separability assumption, introduced in \cite{Don03, Aro12a, Aro12b}.

Let $\M \in \Real^{d \times m}_+$ have an exact NMF
such that $\M = \F\W$
for $\F \in \Real^{d \times k}_+$ and $\W \in \Real^{k \times m}_+$.
Separability assumes that $\M$ can be further written as 
 \begin{equation} \label{Eq: separable matrix} 
  \M = \F\W \ \mbox{for} \ \F \in \Real^{d \times k}_+ \ \mbox{and} \ \W = [\I, \H]\Pib \in \Real^{k \times m}_+.
 \end{equation}
Here,
$\I$ is the \by{k}{k} identity, 
$\H$ is a \by{k}{(m-k)} nonnegative matrix, and 
$\Pib$ is an \by{m}{m} permutation matrix.
We call $\F$ the {\it basis} of $\M$ and $k$ the {\it factorization rank}.
If a nonnegative matrix $\M$ can be written as (\ref{Eq: separable matrix}),
we call it a {\it separable matrix}.
The feature of separable matrices is that all columns of $\F$ appear in those of $\M$.
The separable NMF problem is stated as follows.

\begin{prob}
 Let $\M$ be of the form given as (\ref{Eq: separable matrix}).
 Find a column index set $\IC$ with $k$ elements such that $\M(\IC)$ coincides with the basis $\F$.
\end{prob}
The notation $\M(\IC)$ denotes the submatrix of $\M$ indexed by $\IC$;
in other words, $\M(\IC) = [\m_i : i \in \IC]$ for the $i$th column $\m_i$ of $\M$.
The problem is solvable in polynomial time \cite{Aro12a}.

Separable matrices arise in applications, such as
endmember detection in hyperspectral images \cite{Nas05, Bio12, Ma14, Gil14, Gil15a, Gil15b}
and topic extraction from documents \cite{Aro12a, Aro12b, Aro13, Miz14}.
In such applications, it should be reasonable to assume 
that separability is affected by noise.
Noisy separability assumes that a separable matrix $\M$ given as (\ref{Eq: separable matrix})
is perturbed by $\N \in \Real^{d \times m}$. That is, 
\begin{equation*} 
 \A
  = \M + \N   
  = \F\W + \N \ \mbox{for} \ \F \in \Real^{d \times k}_+, \W = [\I, \H]\Pib \in \Real^{k \times m}_+
  \ \mbox{and} \ \N \in \Real^{d \times m}.
\end{equation*}
If a nonnegative matrix $\A$ can be written as above, 
we call it a {\it noisy separable matrix}.
Consider an algorithm for solving separable NMF problems.
Given a noisy separable matrix $\A$ and factorization rank $k$,
we say that the algorithm is {\it robust to noise} if it can find a column index set $\IC$
with $k$ elements such that $\A(\IC)$ is close to the basis $\F$ of $\A$.

The successive projection algorithm (SPA) was originally proposed by \cite{Ara01} 
in the context of chemometrics.
Recently, Gillis and Vavasis showed in \cite{Gil14} that
SPA works well at solving separable NMF problems.
Given a separable matrix $\M$ and a factorization rank $k$ as input,
SPA finds a column index set $\IC$ with $k$ elements such that
$\M(\IC)$ coincides with the basis $\F$. Even if noise is involved,
it is robust to small perturbations caused by it.
The results in \cite{Gil14} imply that the robustness of SPA can be further improved if
one can make the condition number of the basis in a noisy separable matrix smaller.
Accordingly, Gillis and Vavasis proposed the preconditioned SPA (PSPA) in \cite{Gil15a}.
The results in \cite{Gil15a, Miz16} imply that PSPA is more robust than SPA.

Although PSPA is more robust than SPA,
it has an issue in regard to its computational cost.
To see this clearly, let us  recall how PSPA constructs a conditioning matrix.
The input is $\A \in \Real^{d \times m}$ and a positive integer $k$.
The construction of the conditioning matrix consists of two steps.
The first step computes the top-$k$ truncated singular value decomposition (SVD)
$\A_k = \U_k \Sigmab_k \V_k^\top$ of $\A$
and then constructs $\P = \Sigmab_k \V_k^\top \in \Real^{k \times m}$.
We will explain the decomposition in Section \ref{Subsec: SVD and low-rank approx},
and  $\U_k, \V_k$ and $\Sigmab_k$ are of
the form specified in (\ref{Eq: Uk and Vk}) and (\ref{Eq: Sk}).
The second step computes a \by{k}{k} positive definite matrix $\L^*$
such that $\{\x \in \Real^k : \x^\top \L^* \x \le 1 \}$ is 
a minimum-volume enclosing ellipsoid
for the set of columns $\p_1, \ldots, \p_m$ of $\P$.
This can be obtained by solving a convex optimization problem.
Since $\L^*$ is positive definite, there is a $\C \in \Real^{k \times k}$ such that $\L^* = \C^2$.
PSPA forms $\P^\circ = \C \P$ and runs SPA for $\P^\circ$.

PSPA needs to compute a truncated SVD in the first step
and solve a convex optimization problem in the second step.
In practice, the computational cost of solving the optimization problem
can be sufficiently reduced by employing a cutting plane strategy.
Experimental evidence for this was presented in previous studies \cite{Sun04, Ahi08}.
A truncated SVD can be computed in polynomial time.
Given a \by{d}{m} matrix with $d \le m$,
the truncated SVD can be obtained by constructing the full SVD and then truncating it.
This approach takes $O(d^2m)$.
However, the computation of a truncated SVD is challenging when the matrix is large.
In such cases, the first step is an obstacle to an efficient implementation of PSPA.

\subsection{Our Algorithms and Contributions} \label{Subsec: algorithms and contributions}
The aim of this paper is to develop a modification of PSPA such that
its robustness against noise is high and its computational cost is low.
We modify the first step of PSPA.
The step computes the top-$k$ truncated SVD $\A_k = \U_k \Sigmab_k \V_k^\top$
of an input matrix $\A$.
Regarding $\A_k$,
the Eckart-Young-Mirsky theorem tells us that
\begin{equation*}
  \min_{\scriptsize \mbox{rank}(\B) \le k} \|\A - \B\| = \|\A - \A_k\|
\end{equation*}
holds when the norm is the $L_2$ norm or the Frobenius norm.
A matrix is called a {\it rank-$k$ approximation} to $\A$
if it is the same size as $\A$ and is at most rank $k$.
We can see from the theorem that $\A_k$ is the best rank-$k$ approximation to $\A$ under the norms.
Although PSPA uses $\A_k$,
our modification replaces it with an alternative rank-$k$ approximation.

The key to the development of our modified PSPA is the method for
constructing a rank-$k$ approximation to a matrix.
Ideally, the approximation accuracy should be high and the computational cost should be low.
If the accuracy of the rank-$k$ approximation is close enough to that of the best approximation,
the modified PSPA is expected to be as robust to noise as PSPA.
We incorporate an SPA based rank-$k$ approximation in the modified PSPA.
Algorithm \ref{Alg: SPA based algorithm} describes each step of the algorithm.
$\IC$ of Step 1 is a subset of $\{1, \ldots, m\}$ with $k$ elements 
(Algorithm \ref{Alg: SPA} shows the details of SPA).
Since the rank of $\Q$ is less than or equal to $k$,
the output matrix $\B$ serves as a rank-$k$ approximation to $\A$.
The modified PSPA forms $\P = \Q^\top \A$
by using $\Q$ at Step 3 of Algorithm \ref{Alg: SPA based algorithm};
then, it constructs a minimum-volume enclosing ellipsoid for the columns of $\P$.

\begin{algorithm}
 \caption{SPA based rank-$k$ approximation}
 \label{Alg: SPA based algorithm}
 \smallskip
 Input: $\A \in \Real^{d \times m}$ and integers $q$ and $k$ such that  $0 \le q$ and $0 < k \le \min\{d, m\}$ \\
 Output: $\B \in \Real^{d \times m}$
 \begin{enumerate}[1:]
  \item Run SPA on the input $(\A, k)$ and let $\IC$ be the output.  
  \item Form $\Y = (\A \A^\top)^q \A(\IC) \in \Real^{d \times k}$.

  \item Compute the orthonormal bases of the range space of $\Y$ 
	and form a matrix $\Q$ by stacking them in a column.

  \item Form $\B = \Q\Q^\top \A$ and return it.
 \end{enumerate}
\end{algorithm}

Algorithm \ref{Alg: SPA based algorithm} has a close relationship with
a randomized algorithm in a subspace iteration framework \cite{Rok09, Hal11, Woo14, Gu15}.
The randomized algorithm is sometimes referred to as the randomized subspace iteration.
In contrast to $\Y$ in Step 2,
the randomized algorithm forms $\Y = (\A\A^\top)^q \A\Omegab$ using a Gaussian matrix $\Omegab$.
We will see the relationship between the algorithms in Section \ref{Subsec: relation with the randomized algorithm}.
It should be noted that Algorithm \ref{Alg: SPA based algorithm} is  deterministic,
while the randomized algorithm is probabilistic,
as it exploits the randomness of a Gaussian matrix.

Our first contribution is to provide a theoretical assessment of Algorithm \ref{Alg: SPA based algorithm}.
We derive a bound on the error $\|\A-\B\|_2$ for the input matrix $\A$ and output matrix $\B$.
The main result is Theorem \ref{Theo: main}.
The theorem tells us that,
if $\A$ is noisy separable and the amount of noise is smaller than certain level,
then, the error is close to the best error and decreases toward the best error
as $q$ increases.
We show in Section \ref{Subsec: cost} that
the algorithm takes $O(dmkq)$ arithmetic operations
on $\A \in \Real^{d \times m}$ and integers $q$ and $k$.

Our second contribution is a modification of PSPA
that uses Algorithm \ref{Alg: SPA based algorithm}.
The importance of data preprocessing in SPA has been recognized, and
there are several studies \cite{Nas05, Miz14, Gil15b,Tep16} on it.
Among them, Gillis and Ma proposed a modification of PSPA in \cite{Gil15b}.
It aimed to resolve the cost issue of PSPA.
The basic idea is to skip the computation of an enclosing ellipsoid
and only perform the computation of a truncated SVD.
Hence, our modification differs from their modification. 
One of the favorable features of our modification is that its robustness against noise
increases with $q$.

Our third contribution is an experimental assessment of
Algorithm \ref{Alg: SPA based algorithm} and the modified PSPA.
The results are summarized as follows.

\begin{itemize}
 \item 
       Experiments were conducted on synthetically generated noisy separable matrices.
       When $q$ exceeds $10$,
       the error of Algorithm \ref{Alg: SPA based algorithm} was close to the best error,
       and the modified PSPA significantly improved the robustness to noise compared with SPA,
       even if the amount of noise was large.
       Algorithm \ref{Alg: SPA based algorithm} with $q=10$ often provided
       highly accurate low-rank approximations,
       and the elapsed time was  much shorter than that of a truncated SVD computation.
       These results are displayed in Figures \ref{Fig: approximation error} and \ref{Fig: recovery rate}
       and in Table \ref{Tab: comp time and approx error on synthetic data}.

 \item
      We applied the modified PSPA to a hyperspectral unmixing problem.
      Experiments were conducted on the HYDICE urban data.
      When $q=4$,
      the estimation of endmembers by the modified PSPA coincided with that by PSPA.
      The estimation was close to the spectra of constituent materials identified in the previous study \cite{Zhu14}.
      The experimental results suggested the possibility that
      for hyperspectral unmixing,
      the modified PSPA provides the similar results
      to those of PSPA in less computational time.
      These results are shown in Table \ref{Tab: SAD} and Figure \ref{Fig: abundance maps}.

\end{itemize}
     
The rest of this paper is organized as follows.
Section \ref{Sec: preliminaries} describes the notation and definitions
that are necessary for our discussion.
In addition, it reviews SVD and low-rank approximations to matrices.
Section \ref{Sec: SPA and PSPA} overviews SPA and PSPA
and discusses the computational cost of PSPA. 
Section \ref{Sec: efficient precond} presents the modified PSPA and 
describes the results of our analysis on Algorithm \ref{Alg: SPA based algorithm}.
The details of the analysis are given in Section \ref{Sec: analysis}.
This section also overviews the randomized subspace iteration.
Section \ref{Sec: experiments} reports an empirical study on synthetic and real data.
Section \ref{Sec: concluding remarks} gives a summary and
suggests future research directions.

\section{Preliminaries} \label{Sec: preliminaries}

\subsection{Notation and Definitions}
Let $\Real$ denotes the set of real numbers,
and $\Real^{d \times m}$ the set of \by{d}{m} real matrices.
We use a capital upper-case letter $\A$ to denote a matrix.
A capital lower-case letter with subscript $\a_i$ indicates the $i$th column,
and a lower-case letter with subscript $a_{ij}$ indicates the $(i,j)$th element.
Let $\A \in \Real^{d \times m}$.
We denote by $\A^\top$ and $\mbox{rank}(\A)$ the transpose and rank of $\A$.
The notation $\| \A \|$ stands for the norm of $\A$;
in particular, $\|\A\|_2$ and $\|\A\|_F$ indicate the $L_2$ norm and Frobenius norm.
If the matrix size needs to be specified,
we use the notation $\A_{d \times m}$ to mean that $\A$ is \by{d}{m}.

Let us suppose that $\A$ is \by{d}{m}.
We say that it is {\it diagonal} if $a_{ij} = 0$ for every $i \neq j$.
The element $a_{ii}$ of a diagonal matrix $\A$ is abbreviated as $a_i$.
Let us recall the definition of positive definite and
positive semidefinite matrices.
Suppose that $\A$ is symmetric.
It is {\it positive semidefinite} if $\x^\top \A \x \ge 0$ for every $\x \neq \0$,
and, in particular, it is {\it positive definite} if the inequality holds strictly.
It is known that $\A$ is positive semidefinite (resp. positive definite)
if and only if all the eigenvalues of $\A$ are nonnegative (resp. positive).
Thus, a positive definite matrix is nonsingular.

The notations $\I$, $\0$, $\Pib$, $\e$ and $\e_i$ are used 
to denote the following matrices and vectors;
$\I$ is the identity matrix;
$\0$ is the all-zero matrix;
$\Pib$ is a permutation matrix;
$\e$ is the vector of all ones; and
$\e_i$ is the $i$th unit vector.

\subsection{SVD and Low-Rank Approximations} \label{Subsec: SVD and low-rank approx}
Let $\A \in \Real^{d \times m}$. It can be factorized into
\begin{equation*} 
 \A = \U \Sigmab \V^\top.
\end{equation*}
Here, $\U$ and $\V$ are \by{d}{d} and \by{m}{m} orthogonal matrices,
and $\Sigmab$ is a \by{d}{m} diagonal matrix with nonnegative diagonal elements
$\sigma_1 \ge \cdots \ge \sigma_t \ge 0$ where $t = \min\{d, m\}$.
This factorization is called the {\it singular value decomposition}
and it is commonly abbreviated by SVD.
The diagonal elements $\sigma_1, \ldots, \sigma_t$ are called 
the {\it singular values}.
Obviously, the number of positive singular values coincides with the rank.
We denote by $\sigma_{\mmax}$ and $\sigma_{\mmin}$ the largest and smallest singular values.
Throughout this paper,
we use $\sigma_i$ to denote the $i$th largest singular value of $\A$.
When referring to singular values of several different matrices $\A$ and $\B$, and so on,
to prevent confusion,
the notation $\sigma_i(\A)$ is used to denote the $i$th largest singular value of $\A$.
We define the {\it condition number} of $\A$ as $\sigma_{\mmax}(\A) / \sigma_{\mmin}(\A)$,
and denote it by $\kappa(\A)$.

By picking the top $k$ singular values $\sigma_1, \ldots, \sigma_k$ in $\Sigmab$,
we form a \by{k}{k} diagonal matrix,  
\begin{equation} \label{Eq: Sk} 
 \Sigmab_k =
  \left[
  \begin{array}{ccc}
   \sigma_1 &         &          \\
            & \ddots  &          \\
            &         & \sigma_k \\   
  \end{array}
  \right] \in \Real^{k \times k} 
\end{equation}
and also form 
\begin{equation} \label{Eq: Uk and Vk} 
 \U_k  = [\u_1, \ldots, \u_k] \in \Real^{d \times k} \ \mbox{and} \ \V_k = [\v_1, \ldots, \v_k] \in \Real^{m \times k}
\end{equation}
for the first $k$ columns $\u_1, \ldots, \u_k$ in $\U$ and
those $\v_1, \ldots, \v_k$ in $\V$.
Let
\begin{equation*}
 \A_k = \U_k \Sigmab_k \V_k^\top. 
\end{equation*}
We call it the {\it top-$k$ truncated SVD} of $\A$.
Suppose that $k$ is a positive integer such that $k < \mbox{rank}(\A)$.
Then, 
\begin{eqnarray*}
 & & \min_{\scriptsize \mbox{rank}(\B) \le k} \|\A - \B\|_2  = \|\A - \A_k\|_2 = \sigma_{k+1}, \\
 & & \min_{\scriptsize \mbox{rank}(\B) \le k} \|\A - \B\|_F  = \|\A - \A_k\|_F = \sqrt{\sigma_{k+1}^2 + \cdots + \sigma_{t}^2},
\end{eqnarray*}
where $t = \min\{d,m\}$.
This result is known as the Eckart-Young-Mirsky theorem;
see \cite{Eck36, Tre97, Gol13} for the details.
We can see that $\A_k$ is the best rank-$k$ approximation to $\A$
under the $L_2$ norm error or Frobenius norm error.

In the subsequent discussion, we will use two well-known results on singular values.
The first one is about singular value perturbations,
and the second one is called the interlacing property.

\begin{lemm}[Corollary 8.6.2 of \cite{Gol13}] \label{Lemm: perturbation on singular values}
 Let $\A \in \Real^{d \times m}$ and $\N \in \Real^{d \times m}$.
 Then, $|\sigma_i(\A + \N) - \sigma_i(\A)| \le \|\N\|_2$ holds for $i = 1, \ldots, t$, where $t = \min\{d,m\}$.
\end{lemm}

\begin{lemm}[Theorem 8.1.7 and Corollary 8.6.3 in \cite{Gol13}]  \label{Lemm: interlacing}
 Let $\A \in \Real^{d \times m}$.
 Suppose that $k$ is an integer such that $0 < k \le \min\{d,m\}$.
 Let $\B$ be the \by{d}{k} submatrix of $\A$ 
 consisting of the first $k$ columns.
 Then, $\sigma_k(\A) \ge \sigma_{k}(\B)$ or equivalently $\sigma_{\mmin}(\B)$ holds.
\end{lemm}

In Lemma \ref{Lemm: interlacing},
let $\C$ be the \by{k}{k} submatrix of $\A$  consisting of the first $k$ rows and
first $k$ columns.
Then, the lemma implies $\sigma_k(\A) \ge \sigma_k(\B) = \sigma_k(\B^\top) \ge \sigma_k(\C^\top) = \sigma_k(\C)$,
since $\C^\top$ is the submatrix of $\B^\top$ consisting of the first $k$ columns.

\section{SPA and PSPA} \label{Sec: SPA and PSPA}

\subsection{SPA} \label{Subsec: SPA}

Let us observe a separable matrix from a geometrical point of view.
Let $\M = \F\W$ for $\F \in \Real^{d \times k}_+$ and $\W = [\I,\H] \Pib \in \Real^{k \times m}_+$.
As discussed in \cite{Aro12a},
without loss of generality, every column $\h_i$ of $\H$ can be assumed to satisfy $\e^\top \h_i = 1$.
In addition, we assume that $\mbox{rank}(\F) = k$.
This assumption may be reasonable from the standpoint of a practical application,
because, in such cases, it is less common for the columns of $\F$ to be linearly dependent.
Thus, 
the convex hull of the columns of $\M$ is a $(k-1)$-dimensional simplex in $\Real^d$,
and the vertices correspond to the columns of $\F$.
Accordingly, a separable NMF problem without noise can be restated as follows.
Let $\SC$ be the set of points, including all vertices, in a $(k-1)$-dimensional simplex in $\Real^d$.
Find all points corresponding to the vertices of the simplex from $\SC$.
SPA finds a correct answer for a problem by exploiting 
the property that the maximum of a strongly convex function over the set $\SC$
is attained at one of the vertices.
It uses $f(\x) = \|\x\|_2^2$ and chooses a point to maximize $f$ over $\SC$.
Then, it projects all points of $\SC$ onto the orthogonal space to the chosen point.
Algorithm \ref{Alg: SPA} describes each step of SPA.

\begin{algorithm}
 \caption{SPA \cite{Ara01, Gil14}} \label{Alg: SPA}
 \smallskip
 Input:  $\A \in \Real^{d \times m}$ and an integer $k$ such that $0 < k \le m$ \\
 Output: An index set $\IC$  
 \begin{enumerate}[1:]
   \item Initialize $\S \leftarrow \A $ and $\IC \leftarrow \emptyset$.

   \item Find $i^* = \arg \max_{i=1, \ldots, m} \| \s_i \|_2^2$ 
	 for the columns $\s_1, \ldots, \s_m$ of $\S$.

   \item Set $\t  \leftarrow \s_{i^*}$. Update
	 \begin{equation*}
	  \displaystyle \S \leftarrow \biggl(\I - \frac{\t\t^\top}{\|\t\|_2^2} \biggr) \S
	   \ \mbox{and} \
	   \IC \leftarrow \IC \cup \{i^*\},
	 \end{equation*}
	 
   \item
	If $|\IC| < k$, go back to Step 2; otherwise, return $\IC$ and terminate.

 \end{enumerate}
\end{algorithm}

Now let us describe the results of Gillis and Vavasis  \cite{Gil14}
that SPA works well for solving separable NMF problems.
Let $\M \in \Real^{d \times m}$, and put the following assumption on it.
\begin{assump} \label{Assump}
 $\M$ can be written as
 \begin{equation*}
  \M = \F\W \ \mbox{for} \ \F \in \Real^{d \times k}  \ \mbox{and} \ \W = [\I, \H]\Pib \in \Real^{k \times m}_+.
 \end{equation*}
 Furthermore, $\F$ and the columns $\h_i$ of $\H$ satisfy
 \begin{enumerate}[{\normalfont (a)}]
  \item $\mbox{rank}(\F) = k$.
  \item $\e^\top \h_i \le 1$ for $i = 1, \ldots, m-k$.
\end{enumerate}
\end{assump}
$\W$ of Assumption \ref{Assump} is the same as that of the separability assumption,
whose description is given in (\ref{Eq: separable matrix}).
Meanwhile, $\F$ is more general than that of separability,
since it does not have to be nonnegative.
Gillis and Vavasis  \cite{Gil14} showed that 
SPA for $\M$ satisfying Assumption \ref{Assump} returns $\IC$
such that $\M(\IC)$ coincides with $\F$ of $\M$.
In addition, they showed that it is robust to noise.

\begin{theo}[Theorem 3 of \cite{Gil14}] \label{Theo: robustness of SPA}
 Let $\A = \M + \N$ for $\M \in \Real^{d \times m}$ and $\N \in \Real^{d \times m}$.
 Suppose that $k \ge 2$, $\M$ satisfies  Assumption \ref{Assump},
 and the columns $\n_i$ of $\N$ satisfy $\|\n_i\|_2 \le \epsilon$ for $i = 1, \ldots, m$ with
 \begin{equation*}
  \epsilon < \min \Biggl\{ \frac{1}{2\sqrt{k-1}}, \frac{1}{4} \Biggr\}
   \frac{\sigma_{\mmin}(\F)}{80 \kappa(\F)^2 + 1}.
 \end{equation*}
  Let $\IC$ be the output of Algorithm \ref{Alg: SPA} and 
  $i_1, \ldots, i_k$ be the elements in  $\IC$.
 Then, there is a permutation $\pi : \{1, \ldots, k\} \rightarrow \{1, \ldots, k\}$
 such that
 \begin{equation*}
  \|\a_{i_j} - \f_{\pi(j)}\|_2  \le (80 \kappa(\F)^2 + 1) \epsilon
 \end{equation*}
 for $j = 1, \ldots, k$.
\end{theo}

Although the description of Theorem \ref{Theo: robustness of SPA} 
does not completely match that of Theorem 3 in \cite{Gil14},
Theorem \ref{Theo: robustness of SPA} almost directly follows from Theorem 3 of \cite{Gil14}.
For the details, we refer readers to Remark 1 of \cite{Miz16} and Remark 1 of \cite{Gil15a}.
Theorem \ref{Theo: robustness of SPA} tells us that
SPA can extract the submatrix of a noisy separable matrix close to the basis 
if the amount of noise is small.

In addition to the results mentioned above,
Gillis and Vavasis examined the computational cost of SPA.
Step 3 of Algorithm \ref{Alg: SPA} needs to compute the matrix multiplication of
a \by{d}{d} matrix $\t\t^\top / \|\t\|_2^2$ and a \by{d}{m} matrix $\S$ for every iteration.
The matrix multiplication requires in total $O(d^2mk)$ arithmetic operations.
However, we can reduce these operations by taking into account the equality,
\begin{equation} \label{Eq: useful equality for implementation of SPA} 
 \| ( \I - \b \b^\top ) \a \|_2^2 = \|\a\|_2^2 - (\a^\top\b)^2
  \ \mbox{for} \ \a \in \Real^d \ \mbox{and} \ \b \in \Real^d \ \mbox{with} \ \|\b\|_2 = 1.
\end{equation}
This allows SPA to run in $O(dmk)$ (see Remark 5 of \cite{Gil14}).

As pointed out in \cite{Civ09, Gil14},
SPA is essentially the same as the QR factorization with column pivoting in \cite{Bus65}.
The authors of \cite{Civ09} used SPA in the context of studying 
the complexity of the following problems.
Given $\A \in \Real^{d \times m}$ and a positive integer $k$,
find the set $\IC$ of $k$ column indices among all candidates
to satisfy some criterion.
The authors gave four criteria, including the following two:
(a) $\sigma_{\mmin}(\A(\IC))$ is maximized, and (b)
the volume of parallelepiped spanned by the columns of $\A(\IC)$ is maximized.
They showed the NP-hardness of the problems.
They also used SPA as a heuristic for solving
the volume maximization problem and
studied the accuracy of the heuristic.

\subsection{PSPA} \label{Subsec: PSPA}

The use of a preconditioning matrix makes it possible to improve the robustness performance of SPA.
Theorem \ref{Theo: robustness of SPA} suggests that,
if the condition number of $\F$ can be made close to one,
the allowed noise magnitude range $\| \n_i \|_2$ increases and
the basis gap $\|\a_{i_j} - \f_{\pi(j)} \|_2$ decreases.
For $\A = \F\W + \N \in \Real^{d \times m}$,
we choose some matrix $\C \in \Real^{d \times d}$ and apply it to $\A$
so that $\C \A = \C \F \W + \C \N$.
If $\C$ is chosen to make the condition number of $\C\F$ small,
the robustness of SPA might be improved by performing it on $\C\A$ instead of $\A$,
although the amount of noise could be enlarged by up to a factor $\|\C\|_2$
because $\|\C \N\|_2 \le \|\C\|_2 \|\N\|_2$.
On the basis of this observation,
in \cite{Gil15a}, Gillis and Vavasis proposed PSPA (Algorithm \ref{Alg: PSPA}).

\begin{algorithm}[h]
 \caption{PSPA \cite{Gil15a}} \label{Alg: PSPA}
 \smallskip
 Input:  $\A \in \Real^{d \times m}$ and an integer $k$ such that $0 < k \le \min\{d,m\}$\\
 Output: An index set $\IC$ 
 \begin{enumerate}[1:]
  \item Compute the top-$k$ truncated SVD $\A_k = \U_k \Sigmab_k \V_k^\top$ of $\A$
	and form $\P = \Sigmab_k \V_k^\top \in \Real^{k \times m}$.
	 
  \item Let $\SC = \{\p_1, \ldots, \p_m\}$ for the columns $\p_1, \ldots, \p_m$ of $\P$ and
	compute the optimal solution $\L^*$, which is a \by{k}{k} positive definite matrix,
	of the optimization problem $\MVEE(\SC)$.
	
  \item Compute $\C \in \Real^{k \times k}$ such that $\L^* = \C^2$ and form $\P^\circ = \C\P$.

  \item Run Algorithm \ref{Alg: SPA} on input $(\P^\circ, k)$ and return the output $\IC$.
	
 \end{enumerate}
\end{algorithm}

Here, we explain Steps 1 to 3.
Step 1 computes the top-$k$ truncated SVD $\A_k$ of $\A$,
which is the best rank-$k$ approximation to $\A$, and then forms $\P$.
There is a relationship between $\A_k$ and $\P$ such that 
\begin{equation} \label{Eq: Ak and P} 
 \A_k = \U \left[
      \begin{array}{c}
       \P \\
       \0
      \end{array}
     \right]
\end{equation}
where $\U$ is a \by{d}{d} orthogonal matrix obtained from the SVD $\A = \U \Sigmab \V^\top$ of $\A$
and the leading $k$ columns correspond to the columns of $\U_k$.
Geometrically speaking, since $\U$ is orthogonal,
the columns of $\P$ are a consequence of rotating those of $\A_k$ about the origin.

Step 2 constructs $\SC = \{\p_1, \ldots, \p_m \}$
by gathering $\p_1, \ldots, \p_m$ of $\P \in \Real^{k \times m}$,
and it solves the following optimization problem.
\begin{equation*}
 \MVEE(\SC) :
 \begin{array}{llllll}
  \mbox{minimize}   &  - \log \det(\L) &
  \mbox{subject to} &  \p^\top \L \p \le 1 \  \mbox{for all} \ \p \in \SC, 
                    & \L \succ \0.
 \end{array}
\end{equation*}
$\L$ is a \by{k}{k} symmetric matrix, and this is the decision variable.
Here, the notation $\L \succ \0$ means that 
$\L$ is positive definite, and 
the notation $\log \det(\L)$ stands for
the logarithm of the determinant of $\L$.
$\MVEE(\SC)$ is a problem for constructing an enclosing ellipsoid with the minimum volume for $\SC$.
An ellipsoid centered at the origin in $\Real^k$ is described as
$\{\x \in \Real^k : \x^\top \L \x \le 1 \}$ with a \by{k}{k} positive definite matrix $\L$.
The volume is given as $v / \sqrt{\det (\L)}$ where $v$ is the volume of a unit ball in $\Real^k$
and the value is determined by $k$.
Thus, the optimal solution $\L^*$ of $\MVEE(\SC)$ gives 
a minimum-volume enclosing ellipsoid for the points in $\{\pm \p : \p \in \SC\}$.
If the \by{k}{m} matrix $\P$ has rank $k$,
the convex hull of the points in $\{\pm \p : \p \in \SC\}$ is full-dimensional in $\Real^k$
and its enclosing ellipsoid has a positive volume.
Hence, in that case, the optimal solution of $\MVEE(\SC)$ exists.
$\MVEE(\SC)$ is a convex optimization problem,
and there are efficient algorithms to compute $\L^*$.
For further details on the construction of an enclosing ellipsoid with minimum volume,
we refer readers to \cite{Boy04, Sun04, Ahi08, Miz14} and references therein.

Step 3 computes $\C \in \Real^{k \times k}$ such that $\L^* = \C^2$, and forms $\P^\circ = \C \P$.
Since $\L^*$ is positive definite, such a $\C$ exists and it can be 
constructed by using the eigenvalue decomposition of $\L^*$.
One may be concerned as to how 
this $\C$ serves as a restriction on the condition number of $\F$ in $\A$.
Intuitive explanations are given in Section 2 of \cite{Gil15a} and Section 2.2.1 of \cite{Miz16}.

Gillis and Vavasis showed the robustness of PSPA in \cite{Gil15a}.
A later study \cite{Miz16} was conducted under weaker conditions than those assumed by them.

\begin{theo}[Theorem 3 of \cite{Miz16}] \label{Theo: robustness of precond SPA}
 Let $\A = \M + \N$ for $\M \in \Real^{d \times m}$ and $\N \in \Real^{d \times m}$.
 Suppose that $k \ge 2$, $\M$ satisfies  Assumption \ref{Assump} and 
 $\N$ satisfies $\|\N\|_2 = \epsilon$ with 
 \begin{equation*}
  \epsilon \le \frac{\sigma_{\mmin}(\F)}{1225 \sqrt{k}}.
 \end{equation*}
 Let $\IC$ be the output of Algorithm \ref{Alg: PSPA} and
 $i_1, \ldots, i_k$ be the elements in  $\IC$.
 Then, there is a permutation $\pi : \{1, \ldots, k\} \rightarrow \{1, \ldots, k\}$
 such that
 \begin{equation*}
  \|\a_{i_j} - \f_{\pi(j)}\|_2  \le (432 \kappa(\F) + 4) \epsilon
 \end{equation*}
 for $j = 1, \ldots, k$.
\end{theo}
In comparison with Theorem \ref{Theo: robustness of SPA} describing the robustness of SPA,
this result tells us that
PSPA reduces the bound on the basis gap $\|\a_{i_j} - \f_{\pi(j)}\|_2$
from $\kappa(\F)^2$ to $\kappa(\F)$.

Although PSPA is more robust than SPA,
it has an issue in regard to its computational cost.
In particular, Steps 1 and 2 account for the major part of the cost
for constructing the conditioning matrix $\C$.
Step 1 requires the truncated SVD of a \by{d}{m} matrix $\A$ to be computed.
Suppose that $d \le m$.
The truncated SVD can be obtained by computing the full SVD and then truncating it.
The simple but stable approach takes  $O(d^2m)$; see Chapter 8.6 of \cite{Gol13}.
However, it is challenging when the matrix is large.
In such cases, the truncated SVD computation often takes too long. 
Step 2 requires one to solve $\MVEE(\SC)$
with a \by{k}{k} matrix variable and $m$ linear inequality constraints.
In practice, $\MVEE(\SC)$ can be solved quickly
by performing the interior-point algorithm within a cutting plane strategy \cite{Sun04, Ahi08}
even if $m$ is large; experimental evidence for this was presented in those papers.
The strategy enables us to obtain the optimal solution
by using only some portion of the $m$ constraints.
In addition, the size of a matrix variable  is $k$.
It is usually set to a much smaller integer value than $d$ and $m$.
Accordingly, Step 1 is an obstacle to an efficient implementation of PSPA.

Data preprocessing, such as use of a conditioning matrix,
has often been used in order to improve the robustness of SPA.
Below, we revisit several of the previous studies on preprocessing in SPA.

 \begin{itemize}

  \item ({\it Ellipsoidal Rounding} \cite{Miz14}) 
	The algorithm, called ER-SPA in the paper,
	can be viewed as SPA with preprocessing.
	As PSPA does, 
	it builds an ellipsoid.
	The first two steps are the same as Steps 1 and 2 of Algorithm \ref{Alg: PSPA}, 
	but the ellipsoid built in Step 2 is used for narrowing down the candidates
	for the target column indices of input matrix.
	It picks all points lying on the boundary of the ellipsoid.
	If the number of boundary points exceeds $k$,
	it chooses $k$ points by using SPA.

  \item ({\it Prewhitening and SPA Based Preconditioning} \cite{Gil15b})
	Modifications to resolve the cost issue of PSPA
	were proposed in the paper.
	The prewhitening technique computes
	the top-$k$ truncated SVD $\A_k = \U_k \Sigmab_k \V_k^\top$ of a matrix $\A$
	and constructs a conditioning matrix $\C = \Sigmab_k^{-1} \U_k^\top$.
	The SPA based preconditioning technique 
	avoids the use of the truncated SVD of $\A$
	in favor of the submatrix.
	It performs SPA on the input $(\A, k)$ and extracts $\A(\IC)$
	by using the output $\IC$.
	Then, it computes the truncated SVD of $\A(\IC)$
	and constructs a conditioning matrix in the same way as
	the prewhitening technique.

  \item ({\it SPA with Data Compression} \cite{Tep16})  
	The algorithm uses the randomized subspace iteration \cite{Rok09, Hal11, Woo14, Gu15}
	(Algorithm \ref{Alg: rand subspace iteration}).
	It performs Steps 1 to 3 of Algorithm \ref{Alg: rand subspace iteration} for a matrix $\A$.
	Using a matrix $\Q$ at Step 3, it compresses $\A$ into $\P = \Q^\top \A$.
	The compressed matrix $\P$ has fewer rows than the original one $\A$.
	SPA is applied to $\P$.

\end{itemize}

The first and second algorithms listed above
have been theoretically shown to be robust.
Besides the above algorithms, 
we should also mention the {\it vertex component analysis} in \cite{Nas05}.
This algorithm was developed for hyperspectral unmixing and is widely used
for that purpose.
It can be viewed as an algorithm for solving separable NMF problems with noise
and uses a truncated SVD for data preprocessing.

\section{Efficient Preconditioning for SPA} \label{Sec: efficient precond}

Algorithm \ref{Alg: MPSPA} is our modification of PSPA.

\begin{algorithm}[h]
 \caption{Modified PSPA} \label{Alg: MPSPA}
 \smallskip
 Input:  $\A \in \Real^{d \times m}$ and integers $q$ and $k$ such that $0 \le q$ and $0 < k \le \min\{d,m\}$ \\
 Output: An index set $\IC$ 
 \begin{enumerate}[1:]
  \item Construct $\Q$ 
	by performing Steps 1 to 3 of Algorithm \ref{Alg: SPA based algorithm}
	and form $\P = \Q^\top \A$.
	 
  \item Perform Steps 2 to 4 of Algorithm \ref{Alg: PSPA}.
 \end{enumerate}
\end{algorithm}

Step 1 forms $\P$ using the by-product $\Q$ of Algorithm \ref{Alg: SPA based algorithm}.
The algorithm forms $\B = \Q\Q^\top\A$ for the input matrix $\A$ 
and returns it as output.
The output matrix $\B$ is a rank-$k$ approximation to $\A$.
We can relate $\P$ and $\B$ such that  
\begin{equation} \label{Eq: B and P}
 \B = \Q\Q^\top \A = \Q\P = 
  \wb{\Q} \left[
      \begin{array}{c}
       \P \\
       \0
      \end{array}
     \right]
\end{equation}
where $\wb{\Q}$ is a \by{d}{d} orthogonal matrix whose leading columns are the columns of $\Q$.
Step 1 of Algorithm \ref{Alg: MPSPA} uses a $\P$ derived from the rank-$k$ approximation $\B$
produced by Algorithm \ref{Alg: SPA based algorithm},
while, as can be seen in (\ref{Eq: Ak and P}),
Step 1 of Algorithm \ref{Alg: PSPA} uses a $\P$ derived from the best rank-$k$ approximation $\A_k$.
Hence, if $\B$ is a rank-$k$ approximation to $\A$ with high accuracy,
Algorithm \ref{Alg: MPSPA} should be as robust to noise as Algorithm \ref{Alg: PSPA}.

We need to ensure the existence of the optimal solution of $\MVEE(\SC)$
with $\SC = \{\p_1, \ldots, \p_m\}$ for $\p_1, \ldots, \p_m$ of $\P$.
In particular, it exists if $\mbox{rank}(\B) = k$.
Suppose $\mbox{rank}(\B) = k$.
Since $\mbox{rank}(\P) = \mbox{rank}(\B)$
because of the relationship of (\ref{Eq: B and P}),
we can see that $\mbox{rank}(\P) = k$.
Note here that $\P$ is formed as $\P = \Q^\top \A$.
Since the number of columns of $\Q$ is less than or equal to $k$, so is the number of rows of $\P$.
If it is strictly less than $k$, this contradicts $\mbox{rank}(\P) = k$.
Thus, $\P$ is a \by{k}{m} matrix with rank $k$.
As mentioned in Section \ref{Subsec: PSPA},
this means that the optimal solution of $\MVEE(\SC)$ exists.

 \begin{remark} \label{Remark: robustness of MPSPA}
 Let us go back to the rank-$k$ approximation $\B$ to $\A$
 produced by Algorithm \ref{Alg: SPA based algorithm}.
 If $\B$ is a good approximation to $\A$ and has rank $k$,
 we can show that Algorithm \ref{Alg: MPSPA}
 reaches a similar level of noise-robustness as Algorithm \ref{Alg: PSPA}.
 Suppose that $\B$ satisfies (a) and (b).
\begin{enumerate}[(a)]
 \item $\|\A-\B\|_2 \le 2 \sigma_{k+1}$.
 \item $\mbox{rank}(\B) =k$.
\end{enumerate}
 Then, even if 
 we replace Algorithm \ref{Alg: PSPA} with Algorithm \ref{Alg: MPSPA}
 in the statement of Theorem \ref{Theo: robustness of precond SPA},
 the theorem is still valid as long as we increase 
 the constant values $1225$, $432$, and $4$ involved in it.
 The theorem can be proven by following the arguments
 in the proof of Theorem \ref{Theo: robustness of precond SPA}
 given in Section 3 of \cite{Miz16}.
 \end{remark}

In Section \ref{Subsec: cost},
we will list the computational cost of each step for Algorithm \ref{Alg: SPA based algorithm}.
It reveals that Algorithm \ref{Alg: SPA based algorithm} takes $O(dmkq)$ and
Step 1 of Algorithm \ref{Alg: MPSPA} also takes $O(dmkq)$.
In Section \ref{Subsec: properties},
we will show that the output $\B$ of Algorithm \ref{Alg: SPA based algorithm} has the following properties.

\begin{theo} \label{Theo: main}
 Let $\A = \M + \N$ for $\M \in \Real^{d \times m}$ and $\N \in \Real^{d \times m}$.
 Suppose that $k \ge 2$,  $\M$ satisfies Assumption \ref{Assump}, and $\N$ satisfies
\begin{equation*}
 \|\N\|_2 < \min \Biggl\{ \frac{1}{2\sqrt{k-1}}, \frac{1}{4} \Biggr\}
  \frac{\sigma_{\mmin}(\F)}{1 + 80 \kappa(\F)^2}.
\end{equation*}
 Then, the output $\B$ of Algorithm \ref{Alg: SPA based algorithm} satisfies (a) and (b).
 \begin{enumerate}[{\normalfont (a)}]
  \item $\displaystyle \|\A - \B \|_2 <
	\sigma_{k+1} \sqrt{
	1 + \frac{1}{20164} 
	\Bigl(\frac{\sigma_{k+1}}{\sigma_{k}} \Bigr)^{4q-2}}$.
  \item $\mbox{rank}(\B) = k$.	
 \end{enumerate}
\end{theo}
Result (a) implies
\begin{equation*}
 \|\A - \B \|_2 < \sigma_{k+1} \Biggl(1 + \frac{1}{142} \Bigl(\frac{\sigma_{k+1}}{\sigma_k}\Bigr)^{2q-1}\Biggr).
\end{equation*}
We thus see that the error is bounded by the sum of $\sigma_{k+1}$ and an extra term.
If $\sigma_{k+1} / \sigma_k < 1$,
the extra term decreases toward zero at a rate of $(\sigma_{k+1} / \sigma_k)^2$
as $q$ increases. 
In addition, if a parameter $q$ is chosen as $q \ge 1$,
the result in (a) implies 
$ \| \A - \B \|_2 < \sigma_{k+1} \sqrt{1 + 1/20164}  < 1.00003 \sigma_{k+1}$.
Below, we examine the value of $\sigma_{k+1} / \sigma_k$.

\begin{prop} \label{Prop: size of singular values}
 Let $\A = \M + \N$ for $\M \in \Real^{d \times m}$ and $\N \in \Real^{d \times m}$.
 Suppose that $\M$ satisfies Assumption \ref{Assump} and $\N$ satisfies
 $\| \N \|_2 < \frac{1}{2} \sigma_{\mmin}(\F)$. Then, $\sigma_k > \sigma_{k+1}$.
\end{prop}
\begin{proof}
 From Lemma \ref{Lemm: perturbation on singular values}, we obtain
 the inequality $| \sigma_i(\A) - \sigma_i(\M)| \le \|\N\|_2$.
 When $i=k+1$, it gives
 \begin{equation} \label{Eq: bound on sigma_k+1}
 \sigma_{k+1}(\A) \le \|\N\|_2
 \end{equation}
 since $\sigma_{k+1}(\M) = 0$.
 When $i=k$, it gives
 $\sigma_k(\A) \ge \sigma_k(\M) - \|\N\|_2$.
 From Lemma \ref{Lemm: interlacing}, we have 
 $\sigma_k(\M) = \sigma_k(\F[\I,\H]\Pib) = \sigma_k([\F,\F\H]) \ge \sigma_k(\F)
 \ \mbox{or equivalently} \ \sigma_{\mmin}(\F)$.
 These two inequalities lead to
 \begin{equation} \label{Eq: bound on sigma_k} 
  \sigma_k(\A) \ge \sigma_{\mmin}(\F) - \|\N\|_2.
 \end{equation}
 It follows from (\ref{Eq: bound on sigma_k+1}) and (\ref{Eq: bound on sigma_k}) that
 $\sigma_k(\A) > \sigma_{k+1}(\A)$ if $\|\N\|_2 < \frac{1}{2}\sigma_{\mmin}(\F)$.
\end{proof}
From Proposition \ref{Prop: size of singular values}, we see
that $\sigma_{k+1} / \sigma_k < 1$ holds under the conditions of Theorem \ref{Theo: main},
since $\N$ satisfies $\|\N\|_2 < \frac{1}{324} \sigma_{\mmin}(\F)$ due to $\kappa(\F) \ge 1$.
Hence, our theoretical investigation implies
that Algorithm \ref{Alg: MPSPA} is as robust to noise as Algorithm \ref{Alg: PSPA}
as long as the input matrix of Algorithm \ref{Alg: MPSPA}
satisfies the conditions imposed in Theorem \ref{Theo: main}.
Let $\A = \M + \N$ for $\M \in \Real^{d \times m}$ and $\N \in \Real^{d \times m}$.
Suppose that $\M$ satisfies Assumption \ref{Assump}.
We see from 
Theorem \ref{Theo: main}, Proposition \ref{Prop: size of singular values} and Remark \ref{Remark: robustness of MPSPA}
that, if $\N$ satisfies $\|\N\|_2 = \epsilon$ with
$\epsilon < O \Bigl(\frac{\sigma_{\mmin}(\F)}{\kappa(\F)^2 \sqrt{k} }\Bigr)$, 
the basis gap by Algorithm \ref{Alg: MPSPA} is up to $O(\kappa(\F) \epsilon)$.

We mention a theoretical study on the robustness of SPA with a prewhitening technique,
which was conducted in \cite{Gil15b}.
In the paper, the algorithm is called a prewhitened SPA for short.
The authors put the assumption of $d = k$ on $\M$ in addition to Assumption \ref{Assump},
and then showed the robustness to noise in Theorem 3.4.
Let $\A = \M + \N$ for $\M \in \Real^{d \times m}$ and $\N \in \Real^{d \times m}$.
Suppose that $\M$ satisfies Assumption \ref{Assump} with $d = k$.
The theorem tells us that,
if the columns $\n_i$ of $\N$ satisfies $\|\n_i\|_2 \le \epsilon$ for $i = 1, \ldots, m$
with $\epsilon < O \Bigl( \frac{\sigma_{\mmin}(\F)}{(m - k + 1 )^{3/2} \sqrt{k}} \Bigr)$,
the basis gap by the prewhitened SPA is up to
$O((m - k + 1 )^{3/2} \kappa(\F) \epsilon )$.

\section{SPA Based Rank-$k$ Approximation}\label{Sec: analysis}

 \subsection{Computational Cost of Algorithm \ref{Alg: SPA based algorithm}} \label{Subsec: cost}
Here, we evaluate the computational cost of Algorithm \ref{Alg: SPA based algorithm}.
Recall that the input is $\A \in \Real^{d \times m}$ and integers $q$ and $k$.
Step 1 performs SPA on the input $\A$ and $k$.
As mentioned at the end of Section \ref{Subsec: SPA}, SPA requires $O(dmk)$.
Step 2 forms $\Y = (\A\A^\top)^q \A(\IC)$.
Let $C_{\A}$ denote the arithmetic operation count 
for multiplying $\A$ by an $m$-dimensional vector,
and $C_{\A^\top}$ that for multiplying $\A^\top$ by a $d$-dimensional vector.
As mentioned in Section 4.2 of \cite{Rok09},
when computing it in the order $(\A (\A^\top \cdots (\A (\A^\top \A(\IC)))))$,
the computation requires $kq(C_{\A} + C_{\A^\top})$ arithmetic operations,
written as $O(dmkq)$.
Step 3 computes the orthonormal bases of the range space of $\Y$.
Those can be obtained by constructing the pivoted QR factorization of $\Y$, 
and the construction requires $O(dk^2)$; see Chapter 5.4.2 of \cite{Gol13} for the details.
Step 4 forms $\B = \Q \Q^\top \A$.
When computing it in the order $(\Q (\Q^\top \A))$,
the computation requires $O(dmk)$.
Consequently, we can see that Algorithm \ref{Alg: SPA based algorithm} requires $O(dmkq)$
and Step 1 of Algorithm \ref{Alg: MPSPA} also requires $O(dmkq)$.
Table \ref{Tab: cost of SPA based algorithm} summarizes the computational cost of each step
for Algorithm \ref{Alg: SPA based algorithm}.

 \begin{table}[h]
  \caption{Computational cost of each step for Algorithm \ref{Alg: SPA based algorithm}}
  \label{Tab: cost of SPA based algorithm}
  \centering
  \begin{tabular}{ccccc}
   \toprule
   Step 1   & Step 2  & Step 3    & Step 4    \\
   \midrule
   $O(dmk)$ & $O(dmkq)$ & $O(dk^2)$ & $O(dmk)$  \\
   \bottomrule
  \end{tabular}
 \end{table}

\subsection{Properties of the Rank-$k$ Approximation Produced by Algorithm \ref{Alg: SPA based algorithm}}
\label{Subsec: properties} 

We will start by reviewing the definition of the range space and orthogonal projections.
Let $\W$ be a \by{d}{k} matrix.
A subspace $\{\b = \W \a  :  \a \in \Real^k \}$
in $\Real^d$ is called the {\it range space of $\W$}, and we denote it by $\mbox{range}(\W)$.
Suppose $d \ge k$ in $\W$.
A \by{d}{d} matrix $\P$ is the {\it orthogonal projection} onto $\mbox{range}(\W)$
if $\P = \P^\top$,  $\P^2 = \P$ and $\mbox{range}(\P) = \mbox{range}(\W)$.
We use the notation $\P_{\W}$ to refer to such a matrix $\P$.
From the definition, it is easy to see that $\I - \P_{\W}$ is 
the orthogonal projection onto the complement of $\mbox{range}(\W)$.
If $\W$ has full column rank,
$\P_{\W}$ is given as  $\P_{\W} = \W (\W^\top \W)^{-1} \W^\top$.

Now let us analyze the error $\|\A-\B\|_2$ and the rank of $\B$
for the input matrix $\A \in \Real^{d \times m}$ and output matrix $\B \in \Real^{d \times m}$
of Algorithm \ref{Alg: SPA based algorithm}.
The output matrix $\B$ takes the form $\B = \Q\Q^\top \A$.
The matrix $\Q\Q^\top$ serves as the orthogonal projection onto $\mbox{range}(\Q)$.
In addition, $\mbox{range}(\Q) = \mbox{range}(\Y)$ holds for $\Y \in \Real^{d \times k}$ in Step 2
since the columns of $\Q$ correspond to the orthonormal bases of
the range space of $\Y$.
Hence, 
\begin{equation*}
 \B = \Q\Q^\top \A = \P_{\Y}\A.
\end{equation*}
Using the SVD of $\A$, 
$\A = \U\Sigmab\V^\top$, and the orthogonal matrix $\U$ in the SVD,
we rewrite $\|\A -\B\|_2$ as
\begin{eqnarray*}
 \|\A - \B \|_2
 &=& \| (\I - \P_{\Y})\A \|_2 \\
 &=& \| \U^\top (\I - \P_{\Y}) \U \U^\top \A \|_2 \\
 &=& \| (\I - \U^\top \P_{\Y} \U) \Sigmab \|_2 \\
 &=& \| (\I - \P_{\U^\top \Y}) \Sigmab \|_2. 
\end{eqnarray*}
The last equality uses the fact that
a \by{d}{d} orthogonal matrix $\U$ and a \by{d}{k} matrix $\Y$ with $d \ge k$ satisfy the relationship 
$\U^\top \P_{\Y} \U = \P_{\U^\top\Y}$; see Proposition 8.4 of \cite{Hal11}.
In a similar way to the above, we rewrite $\sigma_k(\B)$ as 
\begin{eqnarray*}
 \sigma_k(\B) 
 &=& \sigma_k(\P_{\Y}\A) \\
 &=& \sigma_k(\U^\top \P_{\Y} \U \U^\top \A) \\
 &=& \sigma_k(\P_{\U^\top \Y} \Sigmab).  
\end{eqnarray*}
Let
\begin{equation*}
 \Z = \U^\top \Y \in \Real^{d \times k}
\end{equation*}
and partition it into two blocks
\begin{equation*}
 \Z =
  \left[
  \begin{array}{c}
   \Z_1 \\
   \Z_2 
  \end{array}
  \right]
\end{equation*}
where $\Z_1$ is \by{k}{k} and $\Z_2$ is \by{(d-k)}{k}.

We can put an upper bound on $\|\A - \B \|_2 = \|(\I - \P_{\Z}) \Sigmab \|_2$
and a lower bound on $\sigma_k(\B) = \sigma_k(\P_{\Z} \Sigmab)$.
Let $\S$ be a \by{d}{d} diagonal matrix such that the $i$th diagonal element $s_i$ is 
\begin{equation} \label{Eq: s_i} 
  s_i =
  \left\{
  \begin{array}{ll}
   \sigma_i, & i = 1, \ldots, t, \\
   0,        & i = t+1, \ldots, d.
  \end{array}
  \right.
\end{equation}
Here, $\sigma_i$ is the $i$th largest singular value of $\A \in \Real^{d \times m}$ and $t = \min\{d,m\}$.
We partition $\S$ into 
\begin{equation} \label{Eq: S1 and S2}
 \S =
  \left[
  \begin{array}{cc}
   \S_1 & \\
        & \S_2 
  \end{array}
  \right] 
\end{equation}
where $\S_1$ is \by{k}{k} and $\S_2$ is \by{(d-k)}{(d-k)}. 

\begin{lemm} \label{Lemm: low-rank approx by CSS}
 Suppose that the \by{k}{k} upper block matrix 
 $\Z_1$ of $\Z$ is nonsingular. Let $\H = \Z_2 \Z_1^{-1}$.
 \begin{enumerate}[{\normalfont (a)}]
  \item $\|\A - \B \|_2^2 = \|(\I - \P_{\Z}) \Sigmab \|_2^2 \le \|\H \S_1\|_2^2 + \sigma_{k+1}^2$.
  \item $\sigma_k(\B) = \sigma_k(\P_{\Z}\Sigmab) \ge \sigma_k((\I + \H^\top\H)^{-1}\S_1)$.
 \end{enumerate}
\end{lemm}
\begin{proof}

 The above discussion shows that the equalities in (a) and (b) hold.
 The inequality in (a) is proven
 by following the arguments in the proof of Theorem 9.1 in \cite{Hal11}.
 We will give the proof in the Appendix to make the discussion self-contained.
 Now let us turn to the inequality in (b).
 Since $\Z_1$ is supposed to be nonsingular,
 we can write $\Z$ as 
 \begin{equation*}
  \Z =
   \left[
   \begin{array}{c}
    \Z_1 \\
    \Z_2
   \end{array}
   \right] =
   \left[
   \begin{array}{c}
    \I \\
    \H
   \end{array}
   \right]\Z_1.
 \end{equation*}
 Then,
 \begin{eqnarray*}
  \P_{\Z}
  &=& \Z (\Z^\top \Z )^{-1} \Z^\top \nonumber \\
  &=&
   \left[
   \begin{array}{cc}
    (\I + \H^\top \H)^{-1}   & (\I + \H^\top \H)^{-1}\H^\top  \\
    \H(\I + \H^\top \H)^{-1} & \H(\I + \H^\top \H)^{-1}\H^\top 
   \end{array}
   \right]. 
 \end{eqnarray*}
 Since $\P_{\Z} \Sigmab$ takes the form
\begin{equation*}
 \P_{\Z} \Sigmab =
  \left[
  \begin{array}{c|c}
   (\I+\H^\top\H)^{-1} \S_1 & \ast \\
   \hline
    \ast                    & \ast
  \end{array}
  \right],
\end{equation*}
 we see that the \by{k}{k} submatrix consisting of the first $k$ rows and
 the first $k$ columns is $(\I+\H^\top\H)^{-1} \S_1$.
 Thus, as is explained after Lemma \ref{Lemm: interlacing},
 this implies the inequality $\sigma_k(\P_{\Z} \Sigmab) \ge \sigma_k((\I+\H^\top\H)^{-1} \S_1)$.
\end{proof}

Using the SVD $\A = \U\Sigmab\V^\top$ of $\A$,
we rewrite $\Y$ as 
\begin{equation*}
 \Y = (\A\A^\top)^q \A(\IC) = \U \S^{2q} \G.
\end{equation*}
Here, let
\begin{equation} \label{Eq: G}
 \G = \U^\top \A(\IC) \in \Real^{d \times k}
\end{equation}
and partition it into 
\begin{equation*}
 \G =
  \left[
  \begin{array}{c}
   \G_1 \\
   \G_2
  \end{array}
  \right]
\end{equation*}
where $\G_1$ is \by{k}{k} and $\G_2$ is \by{(d-k)}{k}.
Then, $\Z$ can be represented using $\S$ and $\G$ as 
\begin{equation*}
 \Z = \U^\top \Y = \S^{2q} \G.
\end{equation*}
Thus, 
\begin{equation} \label{Eq: Z1 and Z2} 
 \Z_1 = \S_1^{2q} \G_1 \ \mbox{and} \  \Z_2 = \S_2^{2q} \G_2.
\end{equation}
In light of  the above representations,
Lemma \ref{Lemm: low-rank approx by CSS} gives the following result.

\begin{prop} \label{Prop: low-rank approx of CSS}
 Suppose  $k$ is chosen such that $k \le \mbox{rank}(\A)$
 and $\IC$ is chosen such that $\G_1$ is nonsingular.
 Then, the output $\B$ satisfies (a) and (b).
 \begin{enumerate}[{\normalfont (a)}]
  \item
       $\displaystyle \|\A - \B \|_2
       \le
       \sigma_{k+1} \sqrt{ 1 + 
       \biggl(\frac{\sigma_{k+1}}{\sigma_k} \biggr)^{4q-2}
       \|\G_1^{-1}\|_2^2 \|\G_2 \|_2^2 }$.
  \item  $\mbox{rank}(\B) = k$.       
 \end{enumerate}
\end{prop}
\begin{proof}
 $\S_1$ is the \by{k}{k} diagonal matrix specified in (\ref{Eq: S1 and S2}).
 The $i$th diagonal element $s_i$ of $\S_1$ corresponds to
 the $i$th largest singular value $\sigma_i$ of $\A$.
 Since $k$ satisfies  $k \le \mbox{rank}(\A)$,
 $\sigma_1, \ldots, \sigma_k$ are positive, and thus, $\S_1$ is nonsingular.
 In addition, $\G_1$ is supposed to be nonsingular.
 Since $\Z_1 = \S_1^{2q} \G_1$ as shown in (\ref{Eq: Z1 and Z2}),
 the nonsingularity of $\S_1$ and $\G_1$ implies the nonsingularity of $\Z_1$.
 Accordingly, we can use Lemma \ref{Lemm: low-rank approx by CSS}.
 We can now show (a).
 Using $\Z_1 = \S_1^{2q} \G_1$ and $\Z_2 = \S_2^{2q} \G_2$ in (\ref{Eq: Z1 and Z2}),
 we rewrite $\H$ as 
 \begin{equation*}
  \H = \Z_2 \Z_1^{-1} = \S_2^{2q} \G_2 \G_1^{-1} \S_1^{-2q}.
 \end{equation*}
  Combining it and Lemma \ref{Lemm: low-rank approx by CSS}(a),
  we obtain 
  \begin{eqnarray*}
   \|\A-\B\|_2
   &\le& \sqrt{\|\H\S_1 \|_2^2 + \sigma_{k+1}^2} \\
   &=&  \sqrt{\|\S_2^{2q} \G_2 \G_1^{-1} \S_1^{-2q+1} \|_2^2 + \sigma_{k+1}^2} \\
   &\le& \sqrt{ \biggl( \|\S_2^{2q}\|_2 \| \G_2\|_2 \| \G_1^{-1}\|_2 \| \S_1^{-2q+1} \|_2 \biggr)^2
    + \sigma_{k+1}^2} \\
   &=& \sigma_{k+1} \sqrt{ 1 + \biggl(\frac{\sigma_{k+1}}{\sigma_k} \biggr)^{4q-2}
    \|\G_1^{-1}\|_2^2 \|\G_2 \|_2^2 }.
  \end{eqnarray*}
 Next, we show (b).
 Since $\B = \Q\Q^\top \A$ and $\mbox{rank}(\Q) \le k$,
 we have $\mbox{rank}(\B) \le k$.
 Hence, it is sufficient to show $\mbox{rank}(\B) \ge k$.
 Lemma~\ref{Lemm: low-rank approx by CSS}(b) gives
 the inequality $\sigma_k(\B) \ge \sigma_k((\I+\H\H^\top)^{-1}\S_1)$.
 Here, note that a matrix $(\I+\H\H^\top)^{-1}\S_1$ is  \by{k}{k}.
 As the above discussion shows, 
 since $\S_1$ is nonsingular,  so is the matrix $(\I+\H^\top\H)^{-1}\S_1$.
 Accordingly, $\sigma_k(\B)$ is positive and thus $\mbox{rank}(\B) \ge k$.
 Thus, we conclude that $\mbox{rank}(\B)=k$.
 \end{proof}

Now let us find a lower bound on $\sigma_{\mmin}(\G_1)$ and
an upper bound on $\sigma_{\mmax}(\G_2)$,
since $\|\G_2\|_2 = \sigma_{\mmax}(\G_2)$ and, if $\sigma_{\mmin}(\G_1)$ is positive, 
$\G_1$ is nonsingular and $\|\G_1^{-1}\|_2 = 1 / \sigma_{\mmin}(\G_1)$.

\begin{prop} \label{Prop: bounds on G1 and G2}
 We can bound $\sigma_{\mmin}(\G_1)$ from below and $\sigma_{\mmax}(\G_2)$ from above
 using $\sigma_{k+1}$ and $\sigma_{\mmin}(\A(\IC))$.
 \begin{enumerate}[{\normalfont (a)}]
  \item  $\sigma_{\mmax}(\G_2) \le \sigma_{k+1}$.
  \item  $\sigma_{\mmin}(\G_1) \ge  \max\{0, \sigma_{\mmin}(\A(\IC))  -  \sigma_{k+1}\}$.
 \end{enumerate}
\end{prop}
\begin{proof} 
 We show (a) first. 
 Recall that
 $\G$ in (\ref{Eq: G})
 takes the form $\G = \U^\top \A(\IC) \in \Real^{d \times k}$ and 
 $\G_1$ and $\G_2$ are the \by{k}{k} upper and \by{(d-k)}{k} lower blocks.
 We partition $\U \in \Real^{d \times d}$ into two blocks
 \begin{equation*}
  \U =
   \left[
   \begin{array}{c}
    \U_1 \\
    \U_2
   \end{array}
   \right]
 \end{equation*}
 where $\U_1$ is \by{k}{d} and $\U_2$ is \by{(d-k)}{d}.
 We rewrite $\G_2$ as 
 \begin{eqnarray*} \label{Eq: G2_1}
  \left[
  \begin{array}{c}
   \0 \\
   \G_2
  \end{array}
  \right]
  &=&
  \left[
  \begin{array}{c}
   \0 \\
   \U_2^\top
  \end{array}
  \right]
  \A(\IC) \\
  &=&
  \left[
  \begin{array}{c}
   \0 \\
   \U_2^\top
  \end{array}
  \right]
  \A\E  \\
  &=&
  \left[
  \begin{array}{c}
   \0 \\
   \U_2^\top
  \end{array}
  \right]
  \U \Sigmab \V^\top \E \\
    &=&
  \left[
  \begin{array}{c|c}
   \0_{k \times k}      & \0_{k \times (d-k)} \\
   \hline
   \0_{(d-k) \times k}  & \I_{(d-k) \times (d-k)} 
  \end{array}
  \right]
   \Sigmab \V^\top \E. 
 \end{eqnarray*}
 Here, let $\E = [\e_i : i \in \IC] \in \Real^{m \times k}$ such that $\A(\IC) = \A\E$.
 In the third equality, we have used the SVD $\A = \U\Sigmab \V^\top$ of $\A$.
 Accordingly, $\sigma_{\mmax}(\G_2)$ is bounded from above such that  
 \begin{eqnarray*}
  \sigma_{\mmax}(\G_2) = \|\G_2\|_2
   &=&
   \biggl\| \left[
   \begin{array}{c}
    \0  \\
    \G_2
   \end{array}
   \right] \biggr\|_2 \\
   &=& 
   \biggl\| 
  \left[
  \begin{array}{c|c}
   \0_{k \times k}      & \0_{k \times (d-k)} \\
   \hline
   \0_{(d-k) \times k}  & \I_{(d-k) \times (d-k)} 
  \end{array}
  \right]
   \Sigmab \V^\top \E \biggr\|_2 \\
  &\le&
  \biggl\| 
  \left[
  \begin{array}{c|c}
   \0_{k \times k}      & \0_{k \times (d-k)} \\
   \hline
   \0_{(d-k) \times k}  & \I_{(d-k) \times (d-k)}  
  \end{array}
  \right]
  \Sigmab \biggr\|_2
  \|\V^\top\|_2 \| \E \|_2 = \sigma_{k+1}.
 \end{eqnarray*}
 Next, we show (b).
 Lemma \ref{Lemm: perturbation on singular values} and (a) of this proposition imply
 \begin{eqnarray*}
  \biggl|
   \sigma_i(\G) -
   \sigma_i
   ( \left[
   \begin{array}{c}
    \G_1 \\
    \0
   \end{array}
   \right])
   \biggr|
   =
   | \sigma_i(\G) - \sigma_i(\G_1) |
   \le
   \biggl\|
   \left[
   \begin{array}{c}
    \0 \\
    \G_2
   \end{array}
   \right]\biggr\|_2 = \|\G_2\|_2
   \le \sigma_{k+1}
 \end{eqnarray*}
 for $i=1, \ldots, k$.
 In addition, $\sigma_i(\G) = \sigma_i(\A(\IC))$  holds since $\U$ is orthogonal.
 Accordingly, $\sigma_{\mmin}(\G_1)$ is bounded from below such that 
 $\sigma_{\mmin}(\G_1) \ge \max\{0, \sigma_{\mmin}(\A(\IC)) - \sigma_{k+1} \}$.
\end{proof}

Define $\rho$ as
\begin{equation*}
 \rho = \sigma_{\min}(\A(\IC)) - \sigma_{k+1}. 
\end{equation*}
If $\rho$ is positive, then $\G_1$ is nonsingular and $k \le \mbox{rank}(\A)$.
The former assertion follows directly from Proposition  \ref{Prop: bounds on G1 and G2}(b).
We need to verify the latter one.
Note here that $\IC$ is a set of $k$ column indices in $\A$.
The positivity of $\rho$ gives the inequality $\sigma_{\mmin}(\A(\IC)) > \sigma_{k+1} \ge 0$
and the positivity of $\sigma_{\mmin}(\A(\IC))$ means that
there are $k$ linearly independent columns in $\A$.
Accordingly, $k$ must satisfy $k \le \mbox{rank}(\A)$.
On the basis of the above observation,
we obtain the following corollary from
Propositions \ref{Prop: low-rank approx of CSS} and \ref{Prop: bounds on G1 and G2}.
\begin{coro} \label{Coro: low-rank approx of CSS by rho}
 Suppose that $\IC$ is chosen such that $\rho > 0$.
 Then, the output $\B$ satisfies (a) and (b).
 \begin{enumerate}[{\normalfont (a)}]
  \item $\displaystyle \| \A - \B \|_2
	\le
	\sigma_{k+1} \sqrt{ 1 + \biggl(\frac{\sigma_{k+1}}{\rho} \biggr)^2
	\biggl(\frac{\sigma_{k+1}}{\sigma_k} \Biggr)^{4q-2}}$.
 \item $\mbox{rank}(\B) = k$.	
 \end{enumerate}
\end{coro}
We can see from Corollary \ref{Coro: low-rank approx of CSS by rho} that 
one of the key issues in Algorithm \ref{Alg: SPA based algorithm} 
lies in the method for finding a set $\IC$ of $k$ column indices from $\A$ in Step 1.
In order to reduce the approximation errors of the algorithm,
the corollary tells us that ideally Step 1 should find the $\IC$
among all of the candidates that maximizes $\sigma_{\mmin}(\A(\IC))$.
However, as mentioned in Section \ref{Subsec: SPA}, 
the problem of finding such an $\IC$ is computationally expensive and
has been shown in \cite{Civ09} to be NP-hard.
We therefore use SPA for solving the problem, which  works as a greedy heuristic for it.

Let us examine the value of $\rho = \sigma_{\mmin}(\A(\IC)) - \sigma_{k+1}$
for the output $\IC$ of SPA.
Let $\A = \M + \N$ for $\M \in \Real^{d \times m}$ satisfying Assumption \ref{Assump}
and $\N \in \Real^{d \times m}$.
The singular value perturbation result described in Lemma \ref{Lemm: perturbation on singular values}
gives the inequality $|\sigma_{k+1}(\A) - \sigma_{k+1}(\M)| \le \|\N\|_2$.
Since $\sigma_{k+1}(\M) = 0$, we have $\sigma_{k+1} = \sigma_{k+1}(\A) \le \|\N\|_2$.
SPA has been shown to be robust to noise. If $\|\N\|_2$ is small,
it can find a column index set $\IC$ such that $\A(\IC)$ is close to $\F$ of $\M$.
Accordingly, $\rho$ should be positive if $\|\N\|_2$ is much smaller than $\sigma_{\mmin}(\F)$.
We can justify this argument thanks to 
Theorem \ref{Theo: robustness of SPA} describing the robustness of SPA.

\begin{prop} \label{Prop: rho}
 Let $\A = \M + \N$ for $\M \in \Real^{d \times m}$ and $\N \in \Real^{d \times m}$.
 Suppose that $k \ge 2$,  $\M$ satisfies Assumption \ref{Assump}, and $\N$ satisfies
\begin{equation*}
 \|\N\|_2 < \min \Biggl\{ \frac{1}{2\sqrt{k-1}}, \frac{1}{4} \Biggr\}
  \frac{\sigma_{\mmin}(\F)}{1 + 80 \kappa(\F)^2}.
\end{equation*}
 Then, the output $\IC$ of Algorithm \ref{Alg: SPA} satisfies
 \begin{equation*}
  \rho = \sigma_{\mmin}(\A(\IC)) - \sigma_{k+1}  > \frac{323 - 81 \sqrt{5}}{324} \sigma_{\mmin}(\F) > 0.
 \end{equation*}
\end{prop}
\begin{proof}
 We choose some \by{k}{k} permutation matrix $\Pib$.
 Lemma \ref{Lemm: perturbation on singular values} gives the inequality 
 $| \sigma_{\mmin}(\A(\IC)) - \sigma_{\mmin}(\F \Pib)  | \le \|\D\|_2 $
 wherein we let $\D = \A(\IC) - \F \Pib \in \Real^{d \times k}$.
 Because $\sigma_{\mmin}(\F \Pib)  = \sigma_{\mmin}(\F)$,
 this inequality leads to $\sigma_{\mmin}(\A(\IC)) \ge \sigma_{\mmin}(\F) - \|\D\|_2$.
 As explained above,  $\sigma_{k+1} \le \|\N\|_2$ holds.
 Hence, 
 \begin{eqnarray} \label{Eq: bound on rho} 
  \rho
   = \sigma_{\mmin}(\A(\IC)) - \sigma_{k+1}
   &\ge& \sigma_{\mmin}(\F) - \|\D\|_2 - \|\N\|_2 \nonumber \\
   &\ge& \sigma_{\mmin}(\F) - \sqrt{k} \max_{j=1,\ldots,k}\|\d_j \|_2 - \|\N\|_2.
 \end{eqnarray}
 The second inequality uses the fact that
 the inequality $\|\D\|_2 \le \sqrt{k} \max_{j=1, \ldots, k} \|\d_j \|_2$ holds
 for the columns $\d_j$ of $\D$.
 Here, $\d_j$ is given as $\d_j =  \a_{i_j} - \f_{\pi(j)}$
 for the elements $i_1, \ldots, i_k$ of $\IC$ and
 a permutation $\pi$ corresponding to $\Pib$.

 This proposition supposes
 $\|\N\|_2 < \min\{\frac{1}{2\sqrt{k-1}}, \frac{1}{4}\} \frac{\sigma_{\mmin}(\F)}{1 + 80 \kappa(\F)^2}$.
 Since $\|\n_i\|_2 \le \|\N\|_2$ holds for any column $\n_i$ of $\N$,
 this proposition satisfies all the conditions required in Theorem \ref{Theo: robustness of SPA}.
 Let $\IC$ be the output of SPA and $i_1, \ldots, i_k$ be the elements.
 The theorem ensures that there is a permutation $\pi$ such that
 $\|\d_j\|_2 = \|\a_{i_j} - \f_{\pi(j)}\|_2  \le  \sigma_{\mmin}(\F) \min\{\frac{1}{2\sqrt{k-1}}, \frac{1}{4}\}$
 holds for every $j = 1, \ldots, k$.
 Using it, we can put a bound on $\sqrt{k} \max_{j=1, \ldots, k} \| \d_j \|_2$. That is, 
 \begin{equation*}
  \sqrt{k} \max_{j=1,\ldots, k} \|\d_j\|_2
   \le
    \sigma_{\mmin}(\F) \sqrt{k} \min \biggl\{\frac{1}{2\sqrt{k-1}}, \frac{1}{4} \biggr\}
   =
   \left\{
   \begin{array}{ll}
    \frac{\sqrt{k}}{4} \sigma_{\mmin}(\F)           & \mbox{for} \ k \le 4, \\
    \frac{\sqrt{k}}{2\sqrt{k-1}} \sigma_{\mmin}(\F) & \mbox{for} \ k \ge 5.
   \end{array}
   \right.
 \end{equation*}
 Obviously, $\frac{\sqrt{k}}{4} \le \frac{1}{2}$ for $k \le 4 $.
 Also, $\frac{\sqrt{k}}{2\sqrt{k-1}} \le \frac{\sqrt{5}}{4}$ for $k \ge 5$
 since the function $f(x) = \sqrt{\frac{x}{x-1}}$ is monotonically decreasing  for $x > 1$.
 Therefore, 
 \begin{equation} \label{Eq: bound on D} 
  \sqrt{k} \max_{j=1, \ldots, k} \| \d_j \|_2 \le \frac{\sqrt{5}}{4}\sigma_{\mmin}(\F).
 \end{equation}
 From $\kappa(\F) \ge 1$, we get
 \begin{equation} \label{Eq: bound on N} 
  \|\N\|_2 < \min \biggl\{\frac{1}{2\sqrt{k-1}}, \frac{1}{4}\biggr\}
   \frac{\sigma_{\mmin}(\F)}{1 + 80 \kappa(\F)^2} \le \frac{1}{324}\sigma_{\mmin}(\F).
 \end{equation}
 Combining (\ref{Eq: bound on rho}),  
 (\ref{Eq: bound on D}) and (\ref{Eq: bound on N}),
 we obtain
 \begin{eqnarray*}
  \rho = \sigma_{\mmin}(\A(\IC)) - \sigma_{k+1}
   > \frac{323 - 81\sqrt{5}}{324} \sigma_{\mmin}(\F) > 0.
 \end{eqnarray*}
 The last inequality follows from the fact that 
 $\sigma_{\mmin}(\F)$ is positive given Assumption \ref{Assump}(a).
\end{proof}

Theorem \ref{Theo: main} follows 
from Corollary \ref{Coro: low-rank approx of CSS by rho} and Proposition \ref{Prop: rho}.

\begin{proof}[{\bf (Proof of Theorem \ref{Theo: main})}]
 The conditions supposed in this theorem are the same as those in Proposition \ref{Prop: rho}.
 Thus, we can use the proposition.
 Since it ensures that $\rho$ is positive,
 we can also use Corollary \ref{Coro: low-rank approx of CSS by rho}.
 It follows from (b) of the corollary 
 that the output $\B$ of Algorithm \ref{Alg: SPA based algorithm} satisfies $\mbox{rank}(\B) = k$.
 As explained above, we have $\sigma_{k+1} \le \|\N\|_2$.
 In addition, we have $\|\N\|_2 < \frac{1}{324} \sigma_{\mmin}(\F)$,
 as shown in (\ref{Eq: bound on N}).
 Hence, $\sigma_{k+1} \le \|\N\|_2 < \frac{1}{324} \sigma_{\mmin}(\F)$.
 Since $\rho > \frac{323 - 81\sqrt{5}}{324} \sigma_{\mmin}(\F) > 0$,
 we get 
 \begin{equation*} 
  \frac{\sigma_{k+1}}{\rho} < \frac{1}{323 - 81 \sqrt{5}} < \frac{1}{142}.
 \end{equation*}
 In light of the above inequality, 
 it follows from (a) of the corollary that 
 the output $\B$ of Algorithm \ref{Alg: SPA based algorithm}
 satisfies the inequality of (a) in this theorem.
\end{proof}

\subsection{Relationship with the Randomized Subspace Iteration}
\label{Subsec: relation with the randomized algorithm}

A low-rank matrix approximation plays a fundamental and essential role in data science.
For instance, it serves as one of main tools
in text mining \cite{Pap00, Man08} and collaborative filtering \cite{Can10, Kes10, Gun13}.
A lot of algorithms have been developed to compute it.
Among them, the randomized subspace iteration, studied in \cite{Rok09, Hal11, Men11, Woo14, Gu15},
has attracted growing attention from researchers and practitioners,
since the randomized algorithm is fast and scalable,
and can provide highly accurate low-rank approximations to matrices.
It was proposed by Rokhlin, Szlam and Tygert in \cite{Rok09},
and a further study was made by Halko, Martinsson and Tropp in \cite{Hal11},
Woodruff in \cite{Woo14}, and Gu in \cite{Gu15}.
An empirical study presented by Menon and Elkan in \cite{Men11}
indicates that it is superior in speed and approximation accuracy.

Algorithm \ref{Alg: rand subspace iteration} describes each step of the randomized algorithm
in cases where the oversampling parameter is set to zero.
We explain the parameter in Remark \ref{Remark: randomized subspace iteration}.
Here, a {\it Gaussian matrix}  is a matrix such that
each element is drawn from a Gaussian distribution with mean zero and variance one.
If we replace $\Omegab \in \Real^{m \times k}$ in Step 1
with $\E = [\e_i : i \in \IC] \in \Real^{m \times k}$
by using the output $\IC$ of SPA on the input $\A$ and $k$,
then the output of Algorithm \ref{Alg: rand subspace iteration}
coincides with that of Algorithm \ref{Alg: SPA based algorithm}.

\begin{algorithm}[h]
 \caption{Randomized subspace iteration \cite{Rok09, Hal11, Woo14, Gu15}} \label{Alg: rand subspace iteration}
 \smallskip
 Input: $\A \in \Real^{d \times m}$ and integers $q$ and $k$ such that $0 \le q$
 and $0 < k  \le \min\{d, m\}$ \\
 Output: $\B \in \Real^{d \times m}$
 \begin{enumerate}[1:]
  \item Form an \by{m}{k} Gaussian matrix $\Omegab$.
  \item Form $\Y = (\A \A^\top)^q \A \Omegab \in \Real^{d \times k}$.
  \item Compute the orthonormal bases of the range space of $\Y$ 
	and form a matrix $\Q$ by stacking them in a column.
  \item Form $\B = \Q\Q^\top\A $ and return it.
 \end{enumerate}
\end{algorithm}

The authors in \cite{Rok09, Hal11, Woo14, Gu15} gave a probabilistic analysis of the randomized algorithm.
Theorem 5.8 in \cite{Gu15} shown by Gu tells us that 
the output $\B$ of Algorithm \ref{Alg: rand subspace iteration} satisfies 
\begin{equation*}
 \|\A - \B \|_2  \le \sigma_{k+1} \sqrt{1 + k c^2 \biggl(\frac{\sigma_{k+1}}{\sigma_k} \biggr)^{4q}}
\end{equation*}
with an exception probability of at most $\Delta$.
Here, $c$ is a positive real value determined by $\Delta$, $k$ and $m$,
and it becomes large as $\Delta$ approaches zero.
In common with an error bound in Theorem \ref{Theo: main},
we see that the error is bounded by the sum of $\sigma_{k+1}$ and an extra term, and
the extra term decreases toward zero as $q$ increases,
if $\sigma_{k+1} / \sigma_k < 1$.

Algorithm \ref{Alg: rand subspace iteration} has the same running time
as Algorithm \ref{Alg: SPA based algorithm}.
To see this, note that
the tasks imposed in Algorithm \ref{Alg: rand subspace iteration} are the same as
those in Algorithm \ref{Alg: SPA based algorithm}
except that Algorithm \ref{Alg: rand subspace iteration} constructs $\A \Omegab$
by multiplying $\A \in \Real^{d \times m}$ by $\Omegab \in \Real^{m \times k}$,
while Algorithm \ref{Alg: SPA based algorithm} constructs $\A(\IC)$
by performing SPA on input $\A \in \Real^{d \times m}$ and a positive integer $k$.
The matrix multiplication $\A\Omegab$ requires $O(dmk)$ arithmetic operations. 
This is the same running time of  SPA.
We thus see that Algorithm \ref{Alg: rand subspace iteration} runs in $O(dmkq)$.

\begin{remark} \label{Remark: randomized subspace iteration}
 The algorithm studied in \cite{Rok09, Hal11, Gu15} is more general than
 Algorithm \ref{Alg: rand subspace iteration}.
 It adds an input parameter $\ell$ such that $0 < k \le \ell$,
 and it forms an \by{m}{\ell} Gaussian matrix $\Omegab$
 in Step 1 of Algorithm \ref{Alg: rand subspace iteration}.
 In that case, $\Y$ of Step 2 is a \by{d}{\ell} matrix.
 After Step 3, it appends the following two steps.
 \begin{itemize}
  \item Form $\P = \Q^\top \A$ and compute the top-$k$ truncated SVD $\P_k = \U_k \Sigmab_k \V_k^\top$
	of $\P$.
  \item Form $\B = \Q \P_k$ and return it.
\end{itemize}
 The gap $p = \ell - k$ is referred to as an {\it oversampling parameter}.
 The authors in \cite{Rok09, Hal11, Gu15} analyzed the approximation error of this algorithm.
 The analysis suggests that a positive oversampling parameter has the effect
 of reducing the approximation error.
\end{remark}

\section{Experiments} \label{Sec: experiments}
We implemented algorithms discussed so far on MATLAB
to assess their practical performance.
The following is the list of the algorithms.

\begin{itemize}
 \item Algorithms for computing rank-$k$ approximations.
\begin{itemize}
 \item $\spaApprox$: Algorithm \ref{Alg: SPA based algorithm}.

 \item $\randApprox$: Algorithm \ref{Alg: rand subspace iteration}.
       
 \item $\svdApprox$:
       MATLAB command \ {\tt [Uk, Sk, Vk] = svds(A, k, 'L'); Ak = Uk*Sk*Vk'}.
       The obtained matrix $\A_k$ is the best rank-$k$ approximation to an input matrix $\A$.
       
\end{itemize}
 \item Algorithms for solving separable NMF problems.
 \begin{itemize}
  \item $\spa$:	Algorithm \ref{Alg: SPA}. 

  \item $\pspa$: Algorithm \ref{Alg: PSPA}. 

  \item $\mpspa$: Algorithm \ref{Alg: MPSPA}.

  \item $\erspa$:
	ER-SPA \cite{Miz14}.

  \item $\merspa$: 
	ER-SPA \cite{Miz14} with the first step
	replaced with Step 1 of Algorithm \ref{Alg: MPSPA}.
	
  \item $\spaspa$: SPA based preconditioned SPA \cite{Gil15b}.
	
  \item $\vca$: MATLAB code {\tt vca.m}.
	It is available on the second author's website\footnotemark[1]
	of \cite{Nas05}.
	\footnotetext[1]{\url{http://www.lx.it.pt/~bioucas/code.htm}}
 \end{itemize}
\end{itemize}

As mentioned in Section \ref{Subsec: PSPA},
the first two steps of ER-SPA are the same as Steps 1 and 2 of Algorithm \ref{Alg: PSPA}.
We implemented $\merspa$ by replacing the first step of ER-SPA with 
Step 1 of Algorithm \ref{Alg: MPSPA}.
Below, we describe our implementations of
Algorithms \ref{Alg: SPA based algorithm} to \ref{Alg: rand subspace iteration} and ER-SPA.

 \begin{itemize}
  \item Taking into consideration
	the equality of (\ref{Eq: useful equality for implementation of SPA}),
	we implemented Algorithm \ref{Alg: SPA} such that it runs in $O(dmk)$.
	The implementation is almost the same as the MATLAB code {\tt FastSepNMF.m}.
	It is available on the first author's website\footnotemark[2] of \cite{Gil14}.
	\footnotetext[2]{\url{https://sites.google.com/site/nicolasgillis/code}}

  \item Algorithms \ref{Alg: PSPA} and \ref{Alg: MPSPA} and ER-SPA 
	need to solve the optimization problem $\MVEE(\SC)$.
	Here, we used Algorithm 3 of \cite{Miz14} for solving $\MVEE(\SC)$.
	It performs the interior-point algorithm within a cutting plane strategy.
	We used the interior-point algorithm
	in the MATLAB software package SDPT3 \cite{Toh99b}.

 \item Algorithms \ref{Alg: SPA based algorithm} and \ref{Alg: rand subspace iteration}
       need to compute the orthonormal bases of the range space of $\Y$, given as 
       $\Y = (\A\A^\top)^q\A(\IC)$ or $\Y = (\A\A^\top)^q\A\Omegab$.
       To prevent them from losing numerical accuracy as a result of rounding errors,
       we used Algorithm 4.4 of \cite{Hal11}.

 \end{itemize}

Experiments were conducted on a desktop computer
equipped with an Intel Xeon W3565 processor and 12 GB memory and running MATLAB R2015a.

\subsection{Synthetic Data} \label{Subsec: synthetic data}

Three types of experiment were conducted on data sets consisting of
noisy separable matrices. The matrices were synthetically generated.
The first set of experiments examined how accurate the low-rank approximations of $\spaApprox$
for noisy separable matrices are.
In particular, one may raise a concern in cases
where the amount of noise involved in separable matrices is large,
as our theoretical result shown in Theorem \ref{Theo: main} cannot answer it.
The second set of experiments examined
the questions of whether
$\mpspa$ sufficiently improves the noise-robustness over that of $\spa$
and whether it is as robust as $\pspa$.
These experiments used the same data sets as in the first ones.
We also tried to determine 
whether $\spaApprox$ is useful for preprocessing in $\erspa$.
In parallel with $\mpspa$, we developed $\merspa$.
The third set of experiments examined how long $\spaApprox$ takes.

The data sets consisted of \by{d}{m} noisy separable matrices $\A$
with a factorization rank $k$.
They  were formed as $\A = \F[\I, \H]\Pib + \N$ from
component matrices $\F, \H, \Pib$ and $\N$ generated as follows;
$\F$ was a \by{d}{k} nonnegative matrix and 
the elements were drawn from a uniform distribution on the interval $[0, 1]$;
$\H$ was a \by{k}{(m-k)} nonnegative matrix and 
the columns were drawn from a Dirichlet distribution having $k$ parameters in $[0,1]$;
$\Pib$ was a chosen randomly \by{m}{m} permutation matrix;
and $\N$ was a \by{d}{m} matrix and 
the elements were drawn from a Gaussian distribution with mean zero and variance one.
To control the noise intensity,
we introduced a parameter $\delta$ that took nonnegative real values,
and scaled $\N$ to satisfy $\|\N\|_2 = \delta$.

The first experiments used data sets
consisting of noisy separable matrices of size $(d, m, k) = (500, 100~000, 10)$
with $\delta$ ranging from $0$ to $200$ in increments of $10$.
A single data set consisted of $50$ distinct noisy separable matrices with
$\delta \in \{0, 10, \ldots, 200 \}$, and a total of $21$ data sets were used.
We ran $\spaApprox$ and $\randApprox$ on the data sets
by setting $q$ to $1, 2, 5, 10$ or $15$.
To measure the accuracy of the rank-$k$ approximation $\B$ to $\A$,
we computed the approximation error $\|\A - \B \|_2$.
Note that $\delta$ serves as an upper bound on the best approximation error $\|\A - \A_k \|_2$.
Since $\A$ is of the form $\A = \F[\I, \H]\Pib + \N$ for
$\F \in \Real^{d \times k}_+$, $[\I, \H]\Pib \in \Real^{k \times m}_+$ and $\N \in \Real^{d \times m}$,
and the inner matrix $\F[\I, \H]\Pib$ is a rank-$k$ approximation to $\A$, we have 
\begin{equation*}
 \| \A - \A_k \|_2 \le \|\A - \F[\I,\H]\Pib \|_2 = \|\N\|_2 \le \delta
\end{equation*}
for the best rank-$k$ approximation $\A_k$ to $\A$.
Figure~\ref{Fig: approximation error} displays the average approximation error 
of $\spaApprox$ and $\randApprox$ on $50$ noisy separable matrices
for each $\delta \in \{0, 10, \ldots, 200\}$.
The black dotted line in the figures connects the points $(\delta, \delta)$
and draws an upper bound on the best approximation errors.
We can see from the figures
that the approximation errors of both algorithms decrease as $q$ increases, 
and they are close to the best approximation errors when $q$ exceeds $10$.
Unlike $\randApprox$, when $\delta$ is less than around $100$,
the approximation errors of $\spaApprox$ remain close to the best ones 
even though $q$ is small, such as $1$ and $2$.
These experimental results imply that $\spaApprox$ with a $q$ larger than $10$
provides highly accurate rank-$k$ approximations for noisy separable matrices
even if the amount of noise is large.

\begin{figure}[h]
 \centering
 \begin{minipage}{.48\linewidth}
  \includegraphics[width=\linewidth]{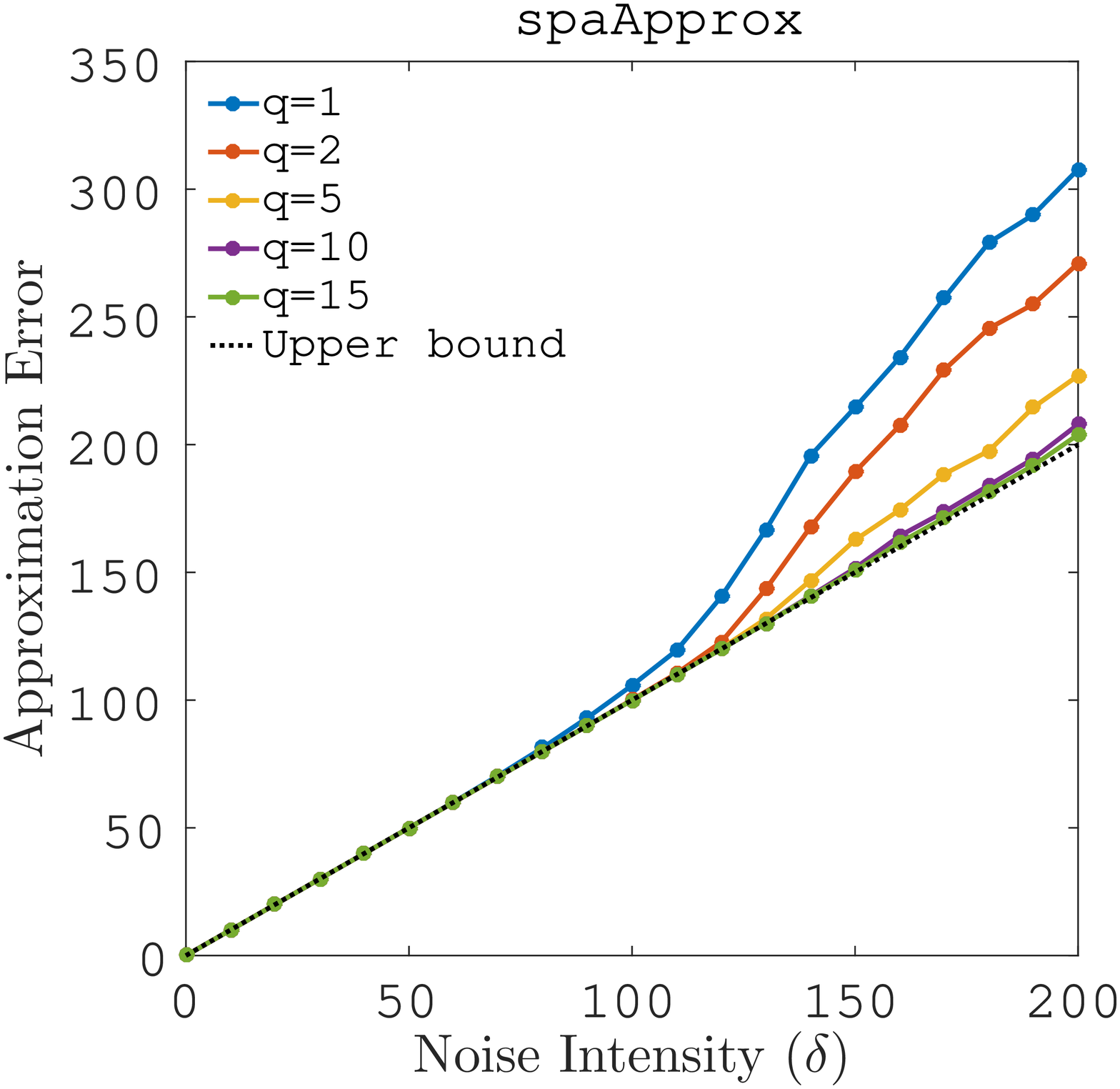}
 \end{minipage}
 \begin{minipage}{.48\linewidth}
  \includegraphics[width=\linewidth]{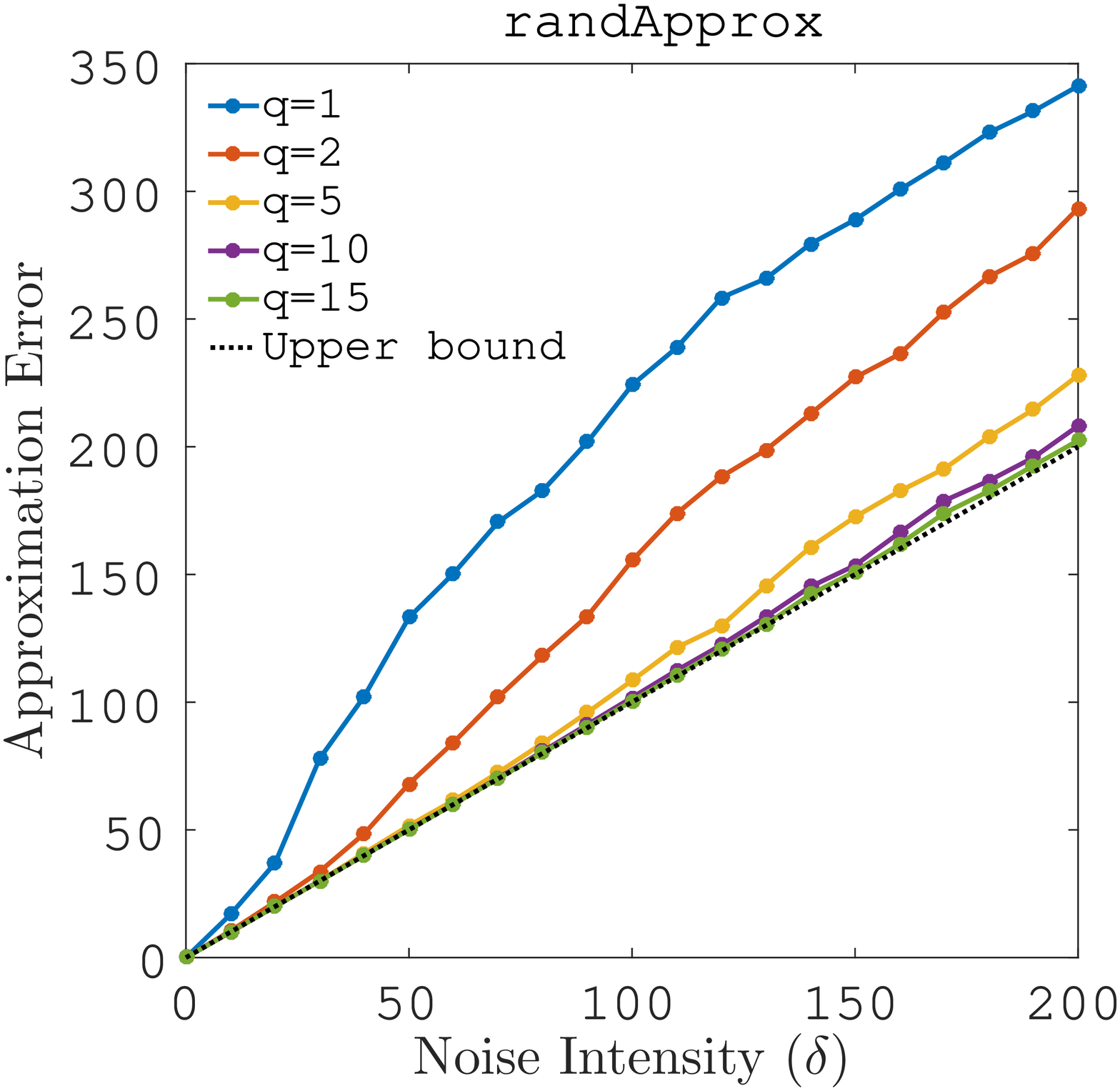}
 \end{minipage}
 \caption{(Results of the first experiments)
 Approximation error of $\spaApprox$ (left) and $\randApprox$ (right)
 for noisy separable matrices of size $(d,m,k) = (500, 100~000, 10)$
 with $\delta \in \{0,10, \ldots, 200\}$.} \label{Fig: approximation error}
\end{figure}

The second experiments ran $\mpspa$, $\spa$, $\pspa$, $\merspa$ and $\erspa$
on the same data sets as in the first experiments.
In the runs of $\mpspa$ and $\merspa$,
$q$ was set from $1,2,5,10$ to $15$.
To measure the robustness of each algorithm,
we computed the recovery rate $\frac{1}{k}|\IC \cap \IC^* |$.
Here, $\IC$ is the index set returned by the algorithm and 
$\IC^*$ is the set of column indices in $\A$
that correspond to the columns of $\F$.
Figure~\ref{Fig: recovery rate} displays the average recovery rates of the algorithms.
We can see that the recovery rates of $\mpspa$ increase with $q$.
When $q$ exceeds $10$, the recovery rates of $\mpspa$ approach those of $\pspa$.
These results imply that, even if the amount of noise is large,
$\mpspa$ with a $q$ larger than $10$ is significantly more robust than $\spa$.
The recovery rates of $\merspa$
show a similar tendency to those of $\mpspa$.
We can hence see that $\spaApprox$ is useful for preprocessing in $\erspa$.

\begin{figure}[h]
 \centering
 \begin{minipage}{.48\linewidth}
 \includegraphics[width=\linewidth]{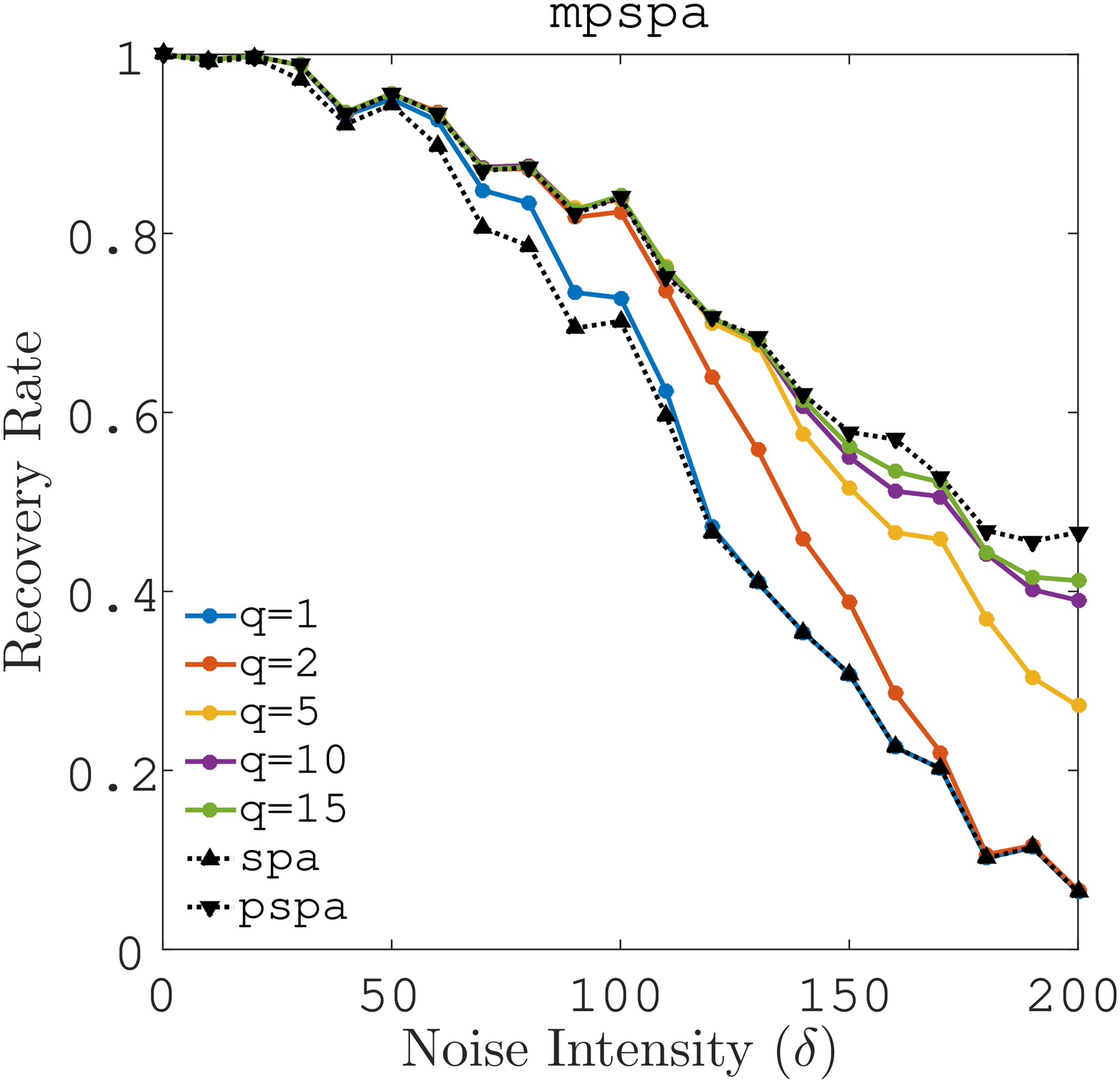}
 \end{minipage}
  \begin{minipage}{.48\linewidth}
  \includegraphics[width=\linewidth]{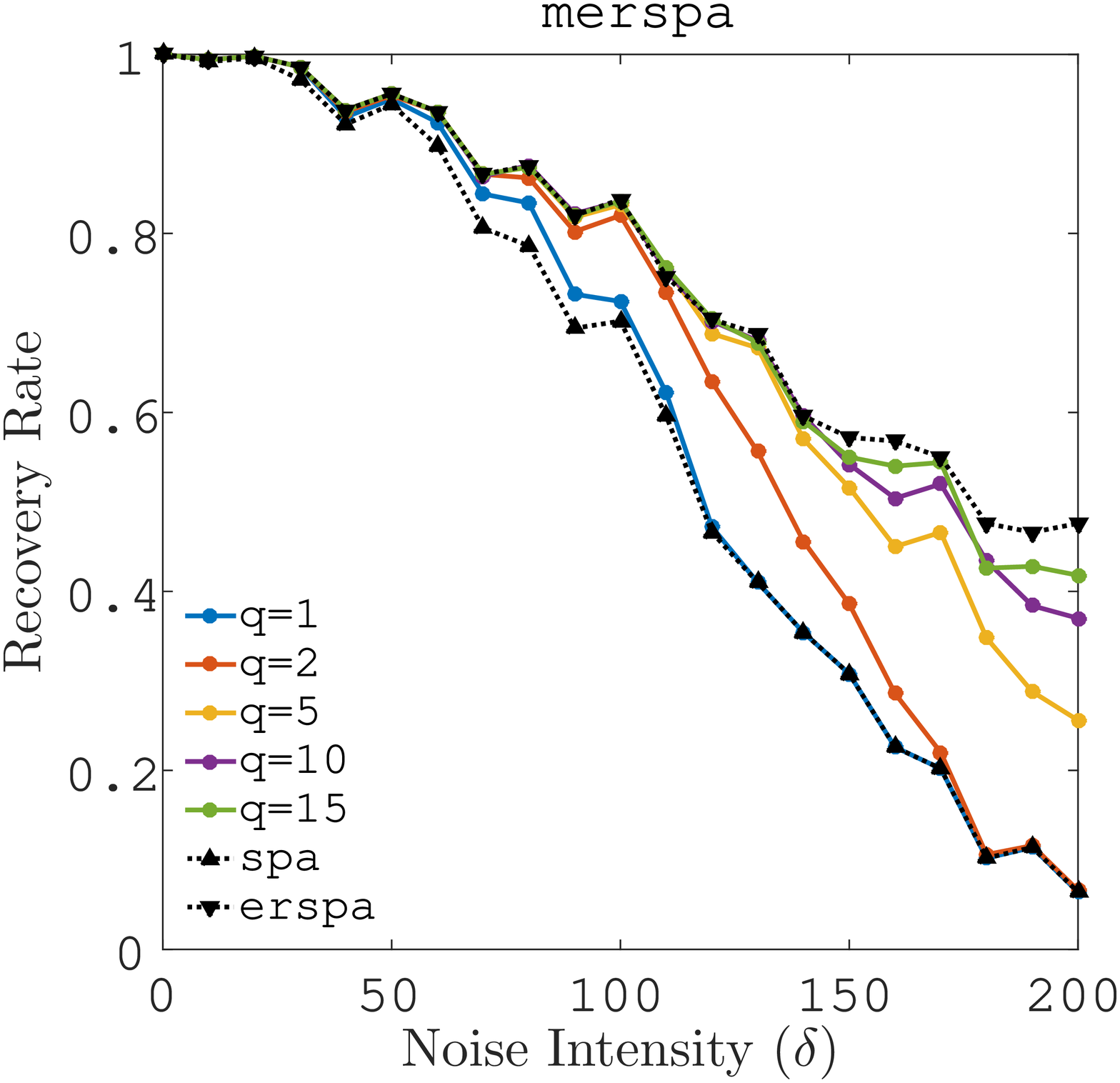}
 \end{minipage}
 \caption{(Results of the second experiments)
 Recovery rate of $\mpspa$ (left) and $\merspa$ (right). 
 The experiments used the same data sets as in the first experiments.}
 \label{Fig: recovery rate}
\end{figure}

The third experiments used two types of data sets:
data sets where $d$ and $k$ were fixed while $m$ was varied such that 
$(d, k) = (500, 10)$ and $m \in \{300~000, 400~000, 500~000\}$,
and data sets where $m$ and $k$ were fixed while $d$ was varied
such that $(m, k) = (100~000, 10)$ and $d \in \{1~000, 2~000, 3~000\}$.
The noise intensity $\delta$ was set as $200$, and the value was
the same for all the data sets.
A single data set consisted of $50$ distinct noisy separable matrices for each $m$ and each $d$,
and a total of $6$ data sets were used.
We ran $\spaApprox$ and $\randApprox$ on the data sets with $q$ set to $10$.
We also ran $\svdApprox$ on them.
We measured the elapsed time of the algorithms for each run
by using the MATLAB commands {\tt tic} and {\tt toc}.
We also evaluated the approximation errors of the algorithms.
Table \ref{Tab: comp time and approx error on synthetic data} summarizes
the average elapsed time in seconds and the average approximation error
on $50$ noisy separable matrices.
The elapsed time is listed in the columns labeled ``time'',
and the average approximation error in those labeled ``error''.
The results for $\spaApprox$ and $\randApprox$ were obtained by setting $q$ to $10$.
We can see that $\spaApprox$ is 18 to 86 times faster than $\svdApprox$ in elapsed time.
The approximation errors of $\spaApprox$ and $\svdApprox$ coincide in the first four digits
for three out of six data sets.
Although the elapsed time of $\spaApprox$ is longer than that of $\randApprox$,
the differences are within a reasonable range.
The approximation errors of $\spaApprox$ are smaller than those of $\randApprox$.

\begin{table}[h]
 \caption{(Results of the third experiments) 
 Elapsed time in seconds and approximation error
 of $\spaApprox$ with $q=10$, $\randApprox$ with $q=10$ and $\svdApprox$.
 The factorization rank $k$ and noise intensity $\delta$
 in the data sets were set to $k = 10$ and $\delta =200$.}
 \label{Tab: comp time and approx error on synthetic data}
 \centering
 \begin{tabular}{rr|rr|rr|rr}
  \toprule
  &       &  \multicolumn{2}{c|}{$\spaApprox$ with $q=10$}
	  &  \multicolumn{2}{c|}{$\randApprox$ with $q=10$}
          &  \multicolumn{2}{c}{$\svdApprox$}  \\
   $d$ & $m$  &  time (s) &  error  & time (s)  & error    & time (s) &  error \\
   \midrule
   500 & 300~000 & 6.0  & $200.12$ & 5.0 & $203.88$ & 109.4 & $199.91$  \\ 
   500 & 400~000 & 8.0  & $199.93$ & 6.8 & $203.70$ & 365.4 & $199.92$  \\
   500 & 500~000 & 10.1 & $199.93$ & 8.7 & $205.87$ & 598.6 & $199.93$ \\
   \midrule
   1~000 & 100~000 & 3.6  & $201.61$ & 3.0 & $204.29$ & 128.5 & $199.86$ \\ 
   2~000 & 100~000 & 7.1  & $200.59$ & 5.9 & $200.95$ & 329.8 & $199.93$  \\
   3~000 & 100~000 & 10.5 & $199.94$ & 8.8 & $202.83$ & 906.8 & $199.93$ \\
   \bottomrule
 \end{tabular}
 \end{table}

\subsection{Real Data -- Application to Hyperspectral Unmixing} \label{Subsec: real data}
Hyperspectral unmixing is a process to identify constituent materials
in a hyperspectral image of a scene.
We shall see that it can be formulated as a separable NMF problem with noise.
The following description follows the tutorials  \cite{Bio12, Ma14}.
A hyperspectral camera is an optical instrument to measure the spectra of materials in a scene.
For instance, the AVIRIS (airborne visible/infrared imaging spectrometer) sensor
developed by the Jet Propulsion Laboratory can scan materials
in 224 spectral bands with wavelengths ranging from 400 nm to 2500 nm.
Let $d$ be the number of spectral bands that a hyperspectral camera can measure.
We associate $\A = [\a_1, \ldots, \a_m] \in \Real^{d \times m}$ with
an image of a scene taken by the camera such that the total number of pixels is $m$.
Here, $\a_i$ stores measured reflectance values at the $i$th pixel,
and the $\ell$th element of $\a_i$ corresponds to the measured value at the $\ell$th band.
A linear mixing model assumes that $\a_1, \ldots, \a_m$ are generated by 
\begin{equation*}
 \a_i = \sum_{j=1}^{k} w_{ji}\f_j + \n_i, \ i=1, \ldots, m
\end{equation*}
where $\f_j \in \Real^d$ is the spectrum of the $j$th constituent material in the scene,
and the elements of $\f_j$ are usually nonnegative;
$w_{ji} \in \Real$ is the mixing rate of the $j$th material at the $i$th pixel
and satisfies $w_{ji} \ge 0$ and $\sum_{j=1}^{k} w_{ji} = 1$; and $\n_i \in \Real^d$ is noise.
We call $\f_j$ the {\it endmember} of the image and 
$w_{ji}$ the {\it abundance} of the endmember $\f_j$ at the $i$th pixel.
In the linear mixing model, hyperspectral unmixing is a problem of
extracting endmembers $\f_1, \ldots, \f_k$ from $\A$. 
We say that the $j^*$th material has a {\it pure pixel}
if there is a pixel containing only the spectrum $\f_{j^*}$ of the material.
That is, there is an $i \in \{1, \ldots, m\}$ such that $w_{ji} = 1$ for $j = j^*$ and $w_{ji} = 0$ for $j \neq j^*$.
It might be reasonable to assume that every endmember has a pure pixel.
This is called the {\it pure pixel assumption},
and it is the same as the separability assumption explained in Section \ref{Sec: intro}.
Accordingly, hyperspectral unmixing under the pure pixel assumption
is equivalent to solving separable NMF problems with noise.
For a matrix $\A$ associated with an image,
the columns $\a_{i_1}, \ldots \a_{i_k}$ of $\A$
extracted by a separable NMF algorithm are the estimations of the endmembers.
The abundances of $\a_{i_1}, \ldots, \a_{ij}$ at some pixel are obtained
by solving the problem of minimizing a convex quadratic function over a simplex.

We are interested in how well $\mpspa$ works for hyperspectral unmixing.
First of all, we report the results of experiments
that evaluated the accuracy of low-rank approximations by $\spaApprox$
for real hyperspectral images.
The experiments used 6 hyperspectral image data:
Cuprite\footnotemark[3],
DC Mall\footnotemark[4],
Indian Pine\footnotemark[4],
Pavia University\footnotemark[5],
Salinas\footnotemark[5]
and Urban\footnotemark[3].
Table \ref{Tab: hyperspectral image} summarizes the number of spectral bands, pixels
and identified constituent materials in each image data.
We removed water absorption and noisy bands from the original data.
The number of bands  in the table is that of the bands we actually used.
These image data have been well studied and are publicly available
at the websites indicated  in the footnotes.
In particular, we used the Cuprite and Urban data
that had been processed for the experiments reported in \cite{Zhu14}.
\footnotetext[3]{Cuprite and Urban data from the website
(\url{http://www.escience.cn/people/feiyunZHU/index.html})}
\footnotetext[4]{DC Mall and Indian Pine  data from the website
(\url{https://engineering.purdue.edu/~biehl/MultiSpec/hyperspectral.html})}
\footnotetext[5]{Pavia University and Salinas data from the website
(\url{http://www.ehu.eus/ccwintco/index.php?title=Hyperspectral_Remote_Sensing_Scenes})}

\begin{table}[h]
 \caption{Hyperspectral image data used in the experiments.}
 \label{Tab: hyperspectral image}
 \centering
 \begin{tabular}{l|r|rr|r}
  \toprule
  &
  \begin{minipage}{0.16\hsize}	  
   \# bands we actually used $(d)$
   \end{minipage}
  &  \multicolumn{2}{c|}{\# pixels $(m)$} &
  \begin{minipage}{0.135\hsize}
   \# identified constituent materials $(k)$
  \end{minipage}
  \\
  \midrule
  Cuprite 
  & 188  &   47~500   & ($250 \times 190)$    & 12 \\
  DC Mall 
  & 191  &   392~960  & ($1~280 \times 307$)  & 7  \\  
  Indian Pine
  & 202  & 1~644~292  & ($2~678 \times 614$)  & 59 \\  
  Pavia University 
  & 103  &   207~400  & ($610 \times 340$)    & 9  \\  
  Salinas
  & 204  &   111~104  & ($512 \times 217$)    & 16  \\
  Urban 
  & 162  &   94~249   & ($307 \times 307$)    & 4  \\
  \bottomrule
 \end{tabular}
\end{table}

The Cuprite data was taken over a cuprite mining area in Nevada, USA.
The data we used was a subimage of the original one.
It has \by{250}{190} pixels with 188 clean bands, 
and there are 12 minerals in the scene. 
The DC Mall data was taken over the Washington DC Mall, USA.
It has \by{1~280}{307} pixels with 191 clean bands.
The scene contains 7 materials.
The Indian Pine data was taken over the Purdue University Agronomy Farm
in West Lafayette, USA.
We used full North-South AVIRIS flightline data.
%stored in the file ``19920612\_AVIRIS\_IndianPine\_NS-line.tif''.
It has \by{2~678}{614} pixels with 202 clean bands.
Although the original data has 220 bands,
we used the data from which 18 noisy bands (104-108, 150-162) are removed.
The scanned area covers 59 types of agricultural and forest areas.
The Pavia University data was taken over the University of Pavia, Italy.
It has \by{610}{340} pixels with 103 bands,
and the scene contains 9 materials.
Salinas data was taken over Salinas Valley, California, USA.
It has \by{512}{217} with 204 clean bands and
the scene contains 16 types of agricultural areas.
The Urban data has \by{307}{307} pixels with 162 clean bands.
Although the original data has 210 bands,
we used the data from which 48 water absorption and
noisy bands (1-4, 76, 87, 101-111, 136-153, 198-210) are removed.
The scene contains 4 materials.
The Cuprite, Indian Pine and Salinas data were acquired with the AVIRIS sensor;
the DC Mall and Urban data with the HYDICE sensor; and
the Pavia University data with the ROSIS sensor.

We ran $\spaApprox$ and $\svdApprox$ on the $6$ hyperspectral image data.
The parameter $k$ of $\spaApprox$ and $\svdApprox$ was set
to the number of identified constituent materials in each image.
The parameter $q$ of $\spaApprox$ was set to $10$ or $20$.
Since some images demonstrated that
the accuracy of the rank-$k$ approximations by $\spaApprox$ with $q=10$ is not so high,
we increased $q$ from $10$ to $20$.
To measure the accuracy of a rank-$k$ approximation $\B$ to a matrix $\A$,
we computed the relative approximation error $\|\A - \B\|_2 / \|\A\|_2$.

Table \ref{Tab: comp time and approx error on real data} summarizes the experimental results.
The columns with the label ``time'' list the elapsed time in seconds of the algorithms
and those with the label ``rel error'' list the relative approximation error of the algorithms.
We first observe the results for data except Indian Pine.
When $q=10$, $\spaApprox$ is 5 to 9 times faster than $\svdApprox$ in elapsed time.
The relative approximation errors of $\spaApprox$ and $\svdApprox$
coincide in the first four digits for Salinas and Urban, while
there are no small gaps between them for Cuprite and DC Mall;
in particular, only the first digits coincide for DC Mall.
When $q=20$,
the first three digits coincide for Cuprite
and the first two digits coincide for DC Mall.
Even if $q$ increases from 10 to 20,
$\spaApprox$ still maintains an advantage in elapsed time over $\svdApprox$;
it is 5 times faster on Cuprite and DC Mall.
We next delve into the discussion of experiments on Indian Pine.
Although $\spaApprox$ with $q = 10$ and $20$ terminated normally
and returned the output,
the execution of $\svdApprox$ was forced to terminate by MATLAB
before returning the output.
The main reason why the execution was interrupted could be that
it requested a large amount of memory.
Indeed, we succeeded to run $\svdApprox$ on Indian Pine
by using a desktop computer with more memory:
it was equipped with Intel Core i7-5775R processor and 16 GB memory.
The experiments revealed that 
the relative approximation error of $\svdApprox$ for Indian Pine is $4.7510 \times 10^{-4}$.

\begin{table}[h]
 \caption{Elapsed time in seconds and relative approximation error of
 $\spaApprox$ with $q=10$ and $20$ and $\svdApprox$ for hyperspectral image data.
 The symbol ``-'' means that the execution of $\svdApprox$
 was interrupted by MATLAB with an error message and was not terminated normally.}
 \label{Tab: comp time and approx error on real data}
 \centering
 \begin{tabular}{l|rr|rr|rr}
  \toprule
  & \multicolumn{2}{c|}{$\spaApprox$ with $q=10$}
  & \multicolumn{2}{c|}{$\spaApprox$ with $q=20$}
  & \multicolumn{2}{c}{$\svdApprox$} \\
  &  time (s) &  rel error & time (s) &  rel error   &  time (s) & rel error \\
  \midrule
  Cuprite              & 0.7  & $2.1944 \times 10^{-3}$ & 1.2   & $2.1073 \times 10^{-3}$ & 6.9  & $2.1061 \times 10^{-3}$ \\
  DC Mall              & 4.0  & $8.3421 \times 10^{-3}$ & 7.2   & $8.1649 \times 10^{-3}$ & 36.5 & $8.1514 \times 10^{-3}$ \\
  Indian Pine          & 81.1 & $4.9211 \times 10^{-4}$ & 150.0 & $4.8687 \times 10^{-4}$ & -    & -                       \\ 
  Pavia University     & 1.7  & $8.3077 \times 10^{-3}$ & 3.1   & $8.3062 \times 10^{-3}$ & 10.8 & $8.3062 \times 10^{-3}$ \\
  Salinas              & 2.1  & $1.3354 \times 10^{-3}$ & 3.8   & $1.3351 \times 10^{-3}$ & 20.6 & $1.3351 \times 10^{-3}$ \\
  Urban                & 0.6  & $2.2364 \times 10^{-2}$ & 1.1   & $2.2364 \times 10^{-2}$ & 3.3  & $2.2364 \times 10^{-2}$ \\
  \bottomrule
 \end{tabular}
\end{table}

Next, we report the results of experiments examining
the accuracy of the endmembers estimated by $\mpspa$ for a hyperspectral image.
The experiments used the Urban data.
Figure \ref{Fig: urban} displays an RGB image of the data.
The constituent materials in the image scene were examined
in the previous studies \cite{Zhu14, Lu14, Wan15},
and 4 materials were identified: asphalt, grass, tree and roof.
The spectra of those materials are available
from the first author's website of \cite{Zhu14} (footnote 3).
We supposed that each of them was a true endmember in the Urban data.
To measure the accuracy of the estimated endmembers,
we evaluated a spectral angle distance (SAD).
Given a true endmember $\f \in \Real^d$ and an estimated endmember $\wh{\f} \in \Real^d$,
it is computed as $\arccos(\f^\top \wh{\f} / \|\f\|_2 \|\wh{\f}\|_2)$.
SAD takes values between $0$ and $1$.
A small SAD value means that an estimated endmember is close to a true endmember,
while a large SAD value means the opposite.
We set $k$ as $4$ and ran $\mpspa$ on a matrix associated with the Urban data.
For comparison, we also ran $\pspa$, $\spa$, $\spaspa$, and $\vca$.
\begin{figure}[h]
 \centering
 \includegraphics[width=0.4\linewidth]{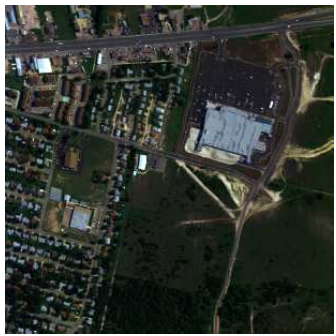}
 \caption{RGB image of Urban data
 generated by the satellite image processing software ENVI.} \label{Fig: urban}
\end{figure}

We examined the output of $\mpspa$ for increasing $q$.
When $q=4$, it coincided with the output of $\pspa$.
At that point,
the relative error of $\spaApprox$ was $2.2587 \times 10^{-2}$,
and there was still a gap between the accuracies of the rank-$k$ approximations by
$\spaApprox$ and $\svdApprox$.
Nevertheless,  $\mpspa$ returned the same output as $\pspa$.
The elapsed time of $\mpspa$ with $q=4$ was 2.3 seconds, while that of $\pspa$ was 6.4 seconds.

Table \ref{Tab: SAD} summarizes the SADs of the algorithms.
The rows correspond to the estimated endmembers,
and values in each row are the SADs for the spectra of the corresponding materials.
We underlined the minimum value on each row.
The estimated endmember is the closest to the spectrum of
a material corresponding to the underlined value.
We can see that the endmembers estimated by $\mpspa$ with $q=4$ and $\pspa$
are close to the spectra of 4 materials, respectively.
However, the estimates of the other algorithms are far from the spectrum of grass.
We computed the abundance maps of true and estimated endmembers.
We let the abundance maps of the true endmembers be the ground truth of the Urban data.
Figure \ref{Fig: abundance maps} displays the ground truth and the abundance maps
obtained by the algorithms.
A pixel color is white when the abundance of the corresponding material is large,
and the color gradually turns to black as the abundance gets smaller.
This enables us to visually confirm that
the abundance maps for $\mpspa$ with $q=4$ and $\pspa$
well match the ground truth.

\begin{table}[h]
 \centering
 \caption{SAD of $\mpspa$ with $q=4$ and $\pspa$ (upper-left),
 $\spa$ (upper-right), $\spaspa$ (lower-left) and $\vca$ (lower-right).
 The rows correspond to the estimated endmembers.
 The underlined value indicates the minimum value on each row.}  
 \label{Tab: SAD}
 \begin{minipage}{.45\linewidth}
  \begin{tabular}{r|rrrr}
   \toprule
   \multicolumn{5}{c}{$\mpspa$ with $q=4$ and $\pspa$} \\
   \midrule
   & asphalt & grass & tree & roof \\
   \midrule
   1  & \underline{0.191} & 0.594 & 0.988 & 0.553 \\
   2 & 0.489 & \underline{0.045} & 0.435 & 0.606 \\
   3 & 0.852 & 0.465 & \underline{0.074} & 0.816 \\
   4 & 0.564 & 0.653 & 0.783 & \underline{0.217} \\
   \bottomrule
  \end{tabular}
 \end{minipage}
 \begin{minipage}{.45\linewidth}
  \begin{tabular}{r|rrrr}
   \toprule
   \multicolumn{5}{c}{$\spa$} \\
   \midrule
   & asphalt & grass & tree & roof \\
   \midrule
   1 & \underline{0.132} & 0.469 & 0.858 & 0.497 \\
   2 & 0.564 & 0.653 & 0.783 & \underline{0.217} \\
   3 & 0.852 & 0.465 & \underline{0.074} & 0.816 \\
   4 & 1.156 & 1.367 & 1.443 & \underline{0.874} \\
   \bottomrule
  \end{tabular}
 \end{minipage}

 \bigskip
 \begin{minipage}{.45\linewidth}
  \begin{tabular}{r|rrrr}
   \toprule
   \multicolumn{5}{c}{$\spaspa$} \\
   \midrule
   & asphalt & grass & tree & roof \\
   \midrule
   1 & \underline{0.191} & 0.594 & 0.988 & 0.553 \\
   2 & 0.564 & 0.653 & 0.783 & \underline{0.217} \\
   3 & 0.852 & 0.465 & \underline{0.074} & 0.816 \\
   4 & 1.156 & 1.367 & 1.443 & \underline{0.874} \\
   \bottomrule
  \end{tabular}
 \end{minipage}
 \begin{minipage}{.45\linewidth}
  \begin{tabular}{r|rrrr}
   \toprule
   \multicolumn{5}{c}{$\vca$} \\
   \midrule
   & asphalt & grass & tree & roof \\
   \midrule
   1 & \underline{0.228} & 0.670 & 1.049 & 0.535 \\
   2 & 0.884 & 0.574 & \underline{0.530} & 0.817  \\
   3 & 0.970 & 0.678 & \underline{0.300} & 0.818  \\
   4 & 1.153 & 1.369 & 1.449 & \underline{0.867}  \\
   \bottomrule
  \end{tabular}
 \end{minipage}
\end{table}

\begin{figure}[p]
 \centering
 \includegraphics[width=1.0\linewidth, bb= 0 220 693 1005]{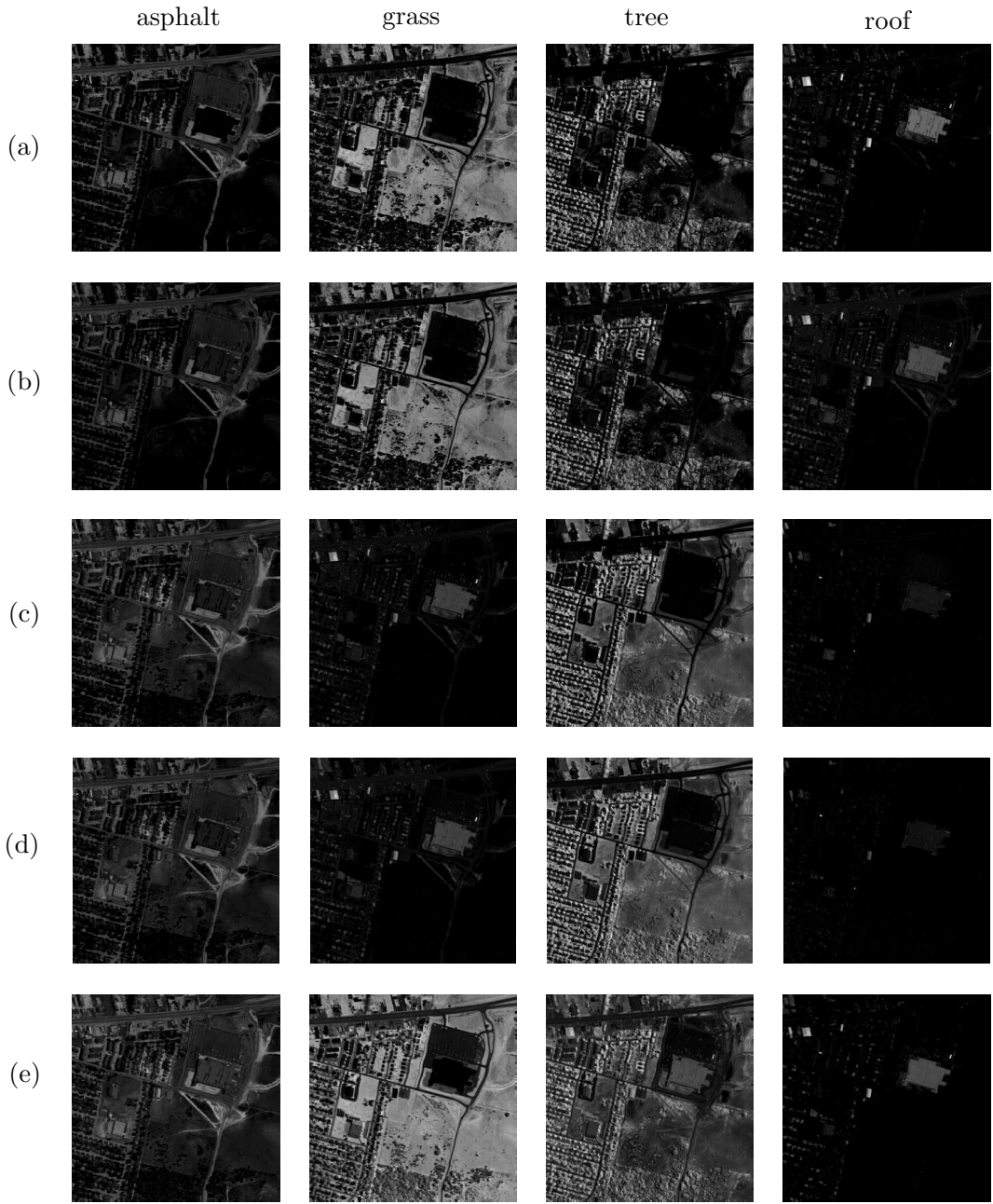}
 \caption{Ground truth and abundance maps obtained by the algorithms: (a) ground truth,
 (b) $\mpspa$ with $q=4$ and $\pspa$, (c) $\spa$, (d) $\spaspa$, and (e) $\vca$.
 From left to right, panels in (a) display the abundance maps of
 the true endmembers, asphalt, grass, tree and roof, respectively.
 From left to right,
 panels in (b) to (e) display the abundance maps of the estimated endmembers
 shown in Table \ref{Tab: SAD} from top to bottom.}
 \label{Fig: abundance maps}
\end{figure}

\section{Summary and Future Research} \label{Sec: concluding remarks}

We have proposed a modification to PSPA, and described it in Algorithm \ref{Alg: MPSPA}.
The modification was motivated by addressing the cost issue of PSPA.
Although PSPA uses the best rank-$k$ approximation to an input matrix,
the modification avoids having to use it and alternatively uses a rank-$k$ approximation
produced by Algorithm \ref{Alg: SPA based algorithm}.
We  evaluated the computational cost of Algorithm \ref{Alg: SPA based algorithm} and
clarified that it is low.
The robustness to noise of Algorithm \ref{Alg: MPSPA}
depends on the approximation accuracy of Algorithm \ref{Alg: SPA based algorithm}.
We derived a bound on the approximation error $\|\A - \B \|_2$ for the input matrix $\A$
and the output matrix $\B$ of Algorithm \ref{Alg: SPA based algorithm}
and described the result in Theorem \ref{Theo: main}.
We conducted an empirical study to assess
the actual performance of Algorithm \ref{Alg: SPA based algorithm} and Algorithm \ref{Alg: MPSPA}.

Finally, we suggest the directions of study for future research.
\begin{itemize}
 \item
      In Theorem \ref{Theo: main}, 
      we put conditions in which an input matrix is noisy separable
      and the amount of noise is small and
      then derived a bound on the approximation error of Algorithm \ref{Alg: SPA based algorithm}.
      Further study is needed to see
      whether it is possible to obtain the error bound under weaker conditions.
      In relation to this,
      it would be interesting to explore
      how well Algorithm \ref{Alg: SPA based algorithm} works for a general matrix
      from theoretical and practical perspectives.

 \item
      Theorem \ref{Theo: main} implies that Algorithm \ref{Alg: SPA based algorithm}
      can produce highly accurate low-rank approximations if
      the value of an input parameter $q$ is set as a large integer.
      However, the theorem may not help us to estimate a parameter value
      required for obtaining such low-rank approximations.
      This is because the theorem 
      describes a bound on the approximation error of Algorithm \ref{Alg: SPA based algorithm}
      by using the ratio between the $(k+1)$th
      and $k$th largest singular values of an input matrix $\A$.
      Regarding Algorithm \ref{Alg: rand subspace iteration},
      the author of \cite{Woo14} has derived the following error bound.
      It is different from an error bound shown in \cite{Gu15} that
      we saw in Section \ref{Subsec: relation with the randomized algorithm}.
      Theorem 4.16 in \cite{Woo14} argues that,
      given a matrix $\A \in \Real^{d \times m}$ and integers $q$ and $k$,
      Algorithm \ref{Alg: rand subspace iteration} returns a rank-$k$ approximation 
      $\B$ to $\A$ satisfying $\|\A - \B\|_2 \le \sigma_{k+1} (c (m-k)/ k )^{1/4q+2}$
      with probability at least $4/5$ where $c$ is a positive real number.
      This theoretical result can help us to estimate the value of $q$
      before running Algorithm \ref{Alg: rand subspace iteration}.
      If we desire to obtain a rank-$k$ approximation $\B$ to $\A$
      satisfying $\|\A - \B\|_2 \le (1+\epsilon) \sigma_{k+1}$,
      the result tells us that $q$ should be set as an integer 
      determined by $\epsilon$, $c$, $m$ and $k$.
      It would be interesting to investigate whether we can obtain this type of an error bound 
      even in case of Algorithm \ref{Alg: SPA based algorithm}.
      Also, further experimental study would be needed
      to observe the relation between the accuracy of low-rank approximation by
      Algorithm \ref{Alg: SPA based algorithm} and a paramenter $q$.

 \item 
       As mentioned in Remark \ref{Remark: randomized subspace iteration},
       the original algorithm description of 
       the randomized subspace iteration \cite{Rok09, Hal11, Gu15}
       includes an input parameter $\ell$.
       Similarly, Algorithm \ref{Alg: SPA based algorithm} can be extended
       to include a parameter $\ell$.
       The extension would probably enable Algorithm \ref{Alg: SPA based algorithm}
       to improve the approximation error by increasing $\ell$ as well as $q$.
       On the other hand, there is a concern that
       the extension involves a computation of a truncated SVD;
       the computational cost becomes large as  $\ell$ increases.
       It would be interesting to see whether 
       the extended algorithm has any advantage in the modification of PSPA.
\end{itemize}

\section*{Acknowledgments}

The authors would like to thank the reviewers for their comments and suggestions
that helped to improve the quality of this paper.
This research was supported by the Japan Society for the Promotion of Science
(JSPS KAKENHI Grant Numbers 15K20986, 26242027).

\appendix
\section*{Appendix \quad Proof of Lemma \ref{Lemm: low-rank approx by CSS}(a)} 

Here, we prove Lemma \ref{Lemm: low-rank approx by CSS}(a).
The proof follows straightforwardly from 
the arguments in the proof of Theorem 9.1 in \cite{Hal11}.
The notation $\A \preceq \B$ below means that $\B - \A$ is positive semidefinite.

\begin{lemm} \label{Lemm: norm inequality on a psd relation}
 If $\A \preceq \B$, then $\|\A\|_2 \le \|\B\|_2$.
\end{lemm}
This inequality follows from the definition of a positive semidefinite matrix
and the property that 
a symmetric matrix $\A$ has the relation $\|\A\|_2 = \max_{\|\x\|_2 = 1}  \x^\top \A \x$.

\begin{lemm}[Proposition 8.3 in \cite{Hal11} and also Lemma 1.1 in \cite{Bou12}]
\label{Lemm: norm inequality on blocks}
 Let $\A$ be a symmetric matrix written in the blocks
 \begin{equation*}
  \A =
   \left[
   \begin{array}{cc}
    \X      & \Z \\
    \Z^\top & \Y
   \end{array}
   \right],
 \end{equation*}
 and suppose that $\A$ is positive semidefinite. Then,  $\|\A\|_2 \le \|\X\|_2 + \|\Y\|_2$.
\end{lemm}

\begin{proof}[{\bf (Proof of Lemma \ref{Lemm: low-rank approx by CSS}(a))}]
 We have $\| (\I - \P_{\Z}) \Sigmab \|_2^2 = \| \Sigmab^\top (\I - \P_{\Z}) \Sigmab \|_2$
 since $\I-\P_{\Z}$ is an orthogonal projection and thus satisfies 
 $\I - \P_{\Z} = (\I - \P_{\Z})^\top $ and $\I - \P_{\Z} = (\I - \P_{\Z})^2$.
 Below, we derive an upper bound on $\| \Sigmab^\top (\I - \P_{\Z}) \Sigmab \|_2$.
 As seen in the proof of Lemma \ref{Lemm: low-rank approx by CSS}(b),
 from the nonsingularity of $\Z_1$,
 we can write $\Z$ as 
 \begin{equation*}
  \Z
   =
  \left[
  \begin{array}{c}
   \Z_1 \\
   \Z_2
  \end{array}
  \right] \\
  =
  \left[
  \begin{array}{c}
   \I \\
   \H 
  \end{array}
  \right] \Z_1
  \ \mbox{and} \
  \H = \Z_2 \Z_1^{-1}.
 \end{equation*}
 Then,
 \begin{equation*}
  \P_{\Z}
  = \Z (\Z^\top \Z )^{-1} \Z^\top 
  =
   \left[
   \begin{array}{cc}
    (\I + \H^\top \H)^{-1}   & (\I + \H^\top \H)^{-1}\H^\top  \\
    \H(\I + \H^\top \H)^{-1} & \H(\I + \H^\top \H)^{-1}\H^\top 
   \end{array}
   \right]. 
 \end{equation*}
 The following matrix inequalities hold.
 \begin{equation*}
  \I - (\I + \H^\top \H)^{-1}             \preceq  \H^\top \H \ \mbox{and} \
  \I - \H (\I + \H^\top \H)^{-1} \H^\top  \preceq  \I. 
 \end{equation*}
 The first one can be checked by considering the SVD of $\H$; see also Proposition 8.2 of \cite{Hal11}.
 The second one comes from the fact that $\H (\I + \H^\top \H)^{-1} \H^\top$ is positive semidefinite.
 From those inequalities, we get 
 \begin{equation*}
  \I - \P_{\Z} \preceq
   \left[
    \begin{array}{cc}
     \H^\top \H                 & - (\I + \H^\top \H)^{-1}\H^\top  \\
     - \H(\I + \H^\top \H)^{-1} & \I
    \end{array}
   \right]
 \end{equation*}
 and this implies 
 \begin{equation} \label{Eq: matrix inequality}
  \Sigmab^\top (\I - \P_{\Z}) \Sigmab \preceq
   \Sigmab^\top \left[
    \begin{array}{cc}
     \H^\top \H                 & - (\I + \H^\top \H)^{-1}\H^\top  \\
     - \H(\I + \H^\top \H)^{-1} & \I
    \end{array}
   \right] \Sigmab.
 \end{equation}
 $\I - \P_{\Z}$ is positive semidefinite,
 since it is an orthogonal projection.
 This means that
 the matrix on the left side of 
 (\ref{Eq: matrix inequality}) is positive semidefinite,
 and hence so is the matrix on the right side.
 The right-side matrix takes the following form.
 If $d \ge m$,
 \begin{equation*}
   \left[
  \begin{array}{c|ccc}
   \S_1 \H^\top \H \S_1 &                 &  \ast  &  \\
   \hline
                        & \sigma_{k+1}^2  &        & \\
       \ast             &                 & \ddots &  \\
                        &                 &        & \sigma_{m}^2 \\
  \end{array}  
   \right] \in \Real^{m \times m}.
 \end{equation*}
 Otherwise,
 \begin{equation*}
  \left[
  \begin{array}{c|ccccccc}
   \S_1 \H^\top \H \S_1   &        &              & \ast  &    &  \\
   \hline
        & \sigma_{k+1}^2  &        &              &   &        &   \\
        &                 & \ddots &              &   &        &   \\
   \ast &                 &        & \sigma_{d}^2 &   &        &   \\
        &                 &        &              & 0 &        &   \\
        &                 &        &              &   & \ddots &   \\
        &                 &        &              &   &        & 0 \\
  \end{array}  
  \right] \in \Real^{m \times m}.
 \end{equation*}
Here, $\S_1$ is the \by{k}{k} upper block of a diagonal matrix $\S$
whose elements are given as (\ref{Eq: s_i});
that is, the diagonal elements $s_1, \ldots, s_k$ of $\S_1$ correspond to 
those $\sigma_1, \ldots, \sigma_k$ of $\Sigmab$.
Accordingly,
from Lemmas \ref{Lemm: norm inequality on a psd relation}
and \ref{Lemm: norm inequality on blocks},
we obtain
$\|(\I-\P_{\Z})\Sigmab \|_2^2 = \|\Sigmab^\top (\I-\P_{\Z})\Sigmab \|_2 \le \|\H\S_1\|_2^2 + \sigma_{k+1}^2$.

\end{proof}

\bibliographystyle{abbrv}
\bibliography{reference}

\end{document}